\pdfoutput=1
\documentclass[final]{siamltex}

\usepackage{amssymb,latexsym,amsfonts,mathrsfs}
\usepackage{amsmath}
\usepackage{mathtools}

\usepackage[ruled,lined,linesnumbered]{algorithm2e}

\usepackage{xcolor}

\usepackage[latin1]{inputenc}
\usepackage{enumitem}
\usepackage{setspace}

\usepackage{microtype}

\usepackage{graphicx}

\numberwithin{equation}{section}

\newcommand{\R}{\mathbb{R}}

\newcommand{\N}{\mathbb{N}}

\newcommand{\mM}{\mathcal{M}}
\newcommand{\mN}{\mathcal{N}}

\newcommand{\mS}{\mathcal{S}}
\newcommand{\mD}{\mathcal{D}}
\newcommand{\mU}{{\mathcal{U}}}
\newcommand{\mV}{\mathcal{V}}

\newcommand{\mO}{\mathcal{O}}

\newcommand{\trans}{\mathsf{T}}

\newcommand{\hR}{\hat{R}}

\newcommand{\bbeta}{\bar{\beta}}
\newcommand{\balpha}{\bar{\alpha}}

\newcommand{\domain}{\mathcal{D}}

\newcommand{\frob}{{\mathsf{F}}}

\DeclareMathOperator{\dist}{dist}

\DeclareMathOperator{\tr}{trace}
\DeclareMathOperator{\ran}{ran}

\DeclareMathOperator*{\argmin}{argmin \,}

\title{Convergence results for projected line-search methods on varieties of low-rank matrices via \L{}ojasiewicz inequality}
\author{Reinhold Schneider\footnotemark[2] \and Andr{\'e} Uschmajew\footnotemark[3]}
\date{}

\begin{document}

\maketitle

\renewcommand{\thefootnote}{\fnsymbol{footnote}}
\footnotetext[2]{Institut f\"ur Mathematik, Technische Universit\"at Berlin, 10623 Berlin, Germany (schneidr@math.tu-berlin.de).
}
\footnotetext[3]{MATHICSE-ANCHP, \'Ecole Polytechnique F\'ed\'erale de Lausanne, 1015 Lausanne, Switzerland. Current address: Hausdorff Center for Mathematics \& Institute for Numerical Simulation, University of Bonn, 53115 Bonn, Germany (uschmajew@ins.uni-bonn.de).}
\renewcommand{\thefootnote}{\arabic{footnote}}

\begin{abstract}
The aim of this paper is to derive convergence results for projected line-search methods on the real-algebraic variety $\mM_{\le k}$ of real $m \times n$ matrices of rank at most $k$. Such methods extend Riemannian optimization methods, which are successfully used on the smooth manifold $\mM_k$ of rank-$k$ matrices, to its closure by taking steps along gradient-related directions in the tangent cone, and afterwards projecting back to $\mM_{\le k}$. Considering such a method circumvents the difficulties which arise from the nonclosedness and the unbounded curvature of $\mM_k$. The pointwise convergence is obtained for real-analytic functions on the basis of a \L{}ojasiewicz inequality for the projection of the antigradient to the tangent cone. If the derived limit point lies on the smooth part of $\mM_{\le k}$, i.e. in $\mM_k$, this boils down to more or less known results, but with the benefit that asymptotic convergence rate estimates (for specific step-sizes) can be obtained without an a priori curvature bound, 
simply from the fact that the limit lies on a smooth manifold. At the same time, one can give a convincing justification for assuming critical points to lie in $\mM_k$: if $X$ is a critical point of $f$ on $\mM_{\le k}$, then either $X$ has rank $k$, or $\nabla f(X) = 0$.
\end{abstract}

\begin{keywords}
Convergence analysis, line-search methods, low-rank matrices, Riemannian optimization, steepest descent, \L{}ojasiewicz gradient inequality, tangent cones
\end{keywords}

\begin{AMS}
65K06, 
40A05, 
26E05, 
65F30, 
15B99, 
15A83, 
\end{AMS}

\pagestyle{myheadings}
\thispagestyle{plain}
\markboth{R. SCHNEIDER AND A. USCHMAJEW}{LINE-SEARCH METHODS ON LOW-RANK MATRIX VARIETIES}

\section{Introduction}

This paper is concerned with line-search algorithms for low-rank matrix optimization. Let $k \le \min(m,n)$. The set 
\[
\mM_k = \{ X \in \R^{m \times n} \vcentcolon \rank(X) = k \}
\]
of real rank-$k$ matrices is a smooth submanifold of $\R^{m \times n}$. Thus, in order to approach a solution of
\begin{equation}\label{eq: smooth min problem}
\min_{X \in \mM_k} f(X),
\end{equation}
where $f \vcentcolon \R^{m \times n} \to \R$ is continuously differentiable,  one can use the algorithms known from Riemannian optimization, the simplest being the  steepest descent method
\begin{equation}\label{eq: gradient flow}
X_{n + 1} = R(X_n, \alpha_n \Pi_{T_{X_n} \mM_k} (- \nabla f(X_n))).
\end{equation}
Here, $\Pi_{T_{X_n} \mM_k}$ is the orthogonal projection on the tangent space at $X_n$, $\alpha_n \ge 0$ is a step-size, and $R$ is a \emph{retraction}, which takes vectors from the affine tangent plane back to the manifold~\cite{AbsilMahonySepulchre2008,Shub1986}. Riemannian optimization on $\mM_k$ (and other matrix manifolds) has become an important tool for low-rank approximation in several applications, e.g. solutions of matrix equations such as Lyapunov equations, model reduction in machine learning, low-rank matrix completion, and others; see, for instance,~\cite{CasonAbsilVanDooren2013,Mishraetal2013,Mishraetal2014,ShalitWeinshallChechik2012,VandereyckenVandewalle2010,Vandereycken2013}. Typically, methods more sophisticated  than steepest descent, such as nonlinear conjugate gradients, Newton's method, or line-search along geodesics, are employed. However, in most cases, convergence results of such line-search methods require the search directions to be sufficiently gradient-related.

An alternative interpretation of the projected gradient method~\eqref{eq: gradient flow} is that of a discretized gradient flow satisfying the Dirac--Frenkel variational principle, i.e., of the integration of the ODE
\[
\dot{X}(t) = \Pi_{T_{X(t)} \mM_k} (- \nabla f(X(t)))
\]
using Euler's explicit method with some step-size strategy. Therefore, our studies are also related to the growing field of dynamical low-rank approximation of ODEs~\cite{KochLubich2007,NonnenmacherLubich2008,LubichOseledets2013} that admit a strict Lyapunov function.

The convergence analysis of sequences in $\mM_k$ is hampered by the fact that this manifold is not closed in the ambient space $\R^{m \times n}$. The manifold properties break down at the boundary which consists of matrices of rank less than $k$. It might happen that a minimizing sequence for~\eqref{eq: smooth min problem} needs to cross such a singular point or even converge to it. Also, the effective domain of definition of a smooth retraction can become tiny at points close to singularities, leading to too small allowed step-sizes in theory. Even if these objections pose no serious problems in practice, they make it difficult to derive a priori convergence statements without making unjustified assumptions on the smallest singular values or adding regularization; cf.~\cite{KressnerSteinlechnerVandereycken2013,KochLubich2007,KochLubich2010,Lubichetal2013,Vandereycken2013}. 

It certainly would be more convenient to optimize and analyze on the closure
\begin{equation}\label{eq: bounded rank matrices}
\mM_{\le k} = \{ X \in \R^{m \times n} \vcentcolon \rank(X) \le k \} 
\end{equation}
of $\mM_k$, which is a real-algebraic variety. In many applications one will be satisfied with any solution of the problem
\begin{equation}\label{eq: min problem on bounded rank}
\min_{X \in \mM_{\le k}} f(X).
\end{equation}
There is no principal difficulty in devising line-search methods on $\mM_{\le k}$. First, in singular points, one has to use search directions in the tangent cone (instead of tangent space), for instance, a projection of the antigradient\footnote{We use this terminology for the negative gradient $- \nabla f$ throughout the paper.} on the tangent cone. The tangent cones of $\mM_{\le k}$ are explicitly known~\cite{CasonAbsilVanDooren2013}, and projecting on them is easy (see Theorem~\ref{prop: form of tangent cone} and Corollary~\ref{cor: projection of negative gradient on tangent cone}). Second, one needs a ``retraction'' that maps from the affine tangent cone back to $\mM_{\le k}$, a very natural choice being a metric projection
\begin{equation}\label{eq: retraction using  best approximation}
R(X_n + \Xi) \in \argmin_{Y \in \mM_{\le k}} \| Y -  (X_n + \Xi) \|_\frob
\end{equation}
(here in Frobenius norm), which can be calculated using singular value decomposition. The aim of this paper is to develop convergence results for such a method based on a \L{}ojasiewicz inequality for the projected antigradient.

Convergence analysis of gradient flows based on the \L{}ojasiewicz gradient inequality~\cite{Lojasiewicz1965}, or on the more general \L{}ojasiewicz--Kurdyka inequality~\cite{Kurdyka1998,Bolteetal2007,Bolteetal2010}, has attracted much attention in nonlinear optimization during recent years~\cite{AbsilMahonyAndrews2005,AttouchBolte2009,Attouchetal2010,AttouchBolteSvaiter2013,BolteDaniilidisLewis2007,CancesEhrlacherLelievre2014,Lageman2002,Lageman2007a,Lageman2007,Levitt2012,LiUschmajewZhang2015,MerletNguyen2013,Noll2014,XuYin2013}. In part, this interest seems to have been triggered by the paper~\cite{AbsilMahonyAndrews2005}, where the following theorem was proved.

\textsc{Theorem.} \textit{
Let $f \vcentcolon \R^N \to \R$ be continuously differentiable, and let $(x_n) \subseteq \R^N$ be a sequence of iterates satisfying the \emph{strong descent conditions}
\begin{gather}
f(x_{n+1}) - f(x_n) \le - \sigma \| \nabla f(x_n) \| \| x_{n+1} - x_n \| \quad \text{(for some $\sigma > 0$)},\label{eq: Armijo linear space} \\
f(x_{n+1}) = f(x_n) \quad \Rightarrow \quad x_{n+1} = x_n.\notag
\end{gather}
Assume also that the sequence possesses a cluster point $x^*$ that satisfies the \L{}ojasiewicz gradient inequality; i.e., there exist $\theta > 0$ and $\Lambda > 0$ such that
\begin{equation}\label{eq: Lojasiewicz linear space}
| f(y) - f(x^*) |^{1-\theta} \le \Lambda \| \nabla f(y) \|
\end{equation}
for all $y$ in some neighborhood of $x^*$. Then $x^*$ is the limit of the sequence $(x_n)$.
}

It is possible to obtain a stronger result if a small step-size safeguard of the form
$\|x_{n+1} - x_n \| \ge \kappa \|\nabla f(x_n)\|$ (for some $\kappa > 0$)
can be assumed. Not only can one then conclude that the limit $x^*$ is a critical point of $f$, but the asymptotic convergence rate in terms of the \L{}ojasiewicz parameters $\theta$ and $\Lambda$ also can be estimated along lines developed, e.g., in~\cite{AttouchBolte2009,BolteDaniilidisLewis2007,Levitt2012,MerletNguyen2013}. No second-order information is required, but a linear convergence rate can only be established when $\theta = 1/2$, which in general cannot be checked in advance. The most notable class of functions satisfying the  \L{}ojasiewicz gradient inequality in every point are real-analytic functions. Therefore, this type of results can be applied to classical line-search algorithms in $\R^N$ when minimizing a real-analytic function using an angle condition for the search directions and Wolfe conditions for the step-size selection~\cite{AbsilMahonyAndrews2005}. 

The theory can be generalized to gradient flows on real-analytic manifolds. Lageman~\cite{Lageman2007a} considered descent iterations on Riemannian manifolds via families of local parametrizations, with retracted line-search methods like~\eqref{eq: gradient flow} being a special case of his setting. Convergence results were obtained by making regularity assumptions on the used family of parametrizations. Merlet and Nguyen~\cite{MerletNguyen2013} considered a discrete projected $\theta$-scheme for integrating an ODE on a smooth embedded submanifold. They proved the existence of step-sizes ensuring convergence to a critical point via \L{}ojasiewicz gradient inequality by assuming a \emph{uniform} bound on the second-order terms in the Taylor expansion of the metric projection, i.e., a curvature bound for the manifold. The main problem with the noncompact submanifold $\mM_k$, without which our work would be unnecessary, is that such an assumption is unjustified. The second-order term in the metric projection 
on $\mM_
k$ scales like the inverse of the 
smallest singular value~\cite{KochLubich2007,AbsilMalick2012}, which gets arbitrarily large in the case when the iterates approach the boundary of $\mM_k$. However, such a uniform bound for the projection is not needed if one is willing to sacrifice some more information on the constants in convergence rate estimates: if a gradient projection method $(x_n)$ on a smooth manifold \emph{is known} to converge to \emph{some} point of it, one will have \emph{some} curvature bound in the vicinity of the limit.

Therefore, our plan is this: via a version of the \L{}ojasiewicz inequality for projections of the antigradients on tangent cones we prove that the iterates of a line-search method on $\mM_{\le k}$ with a particular choice of step-sizes \emph{do} converge. This would not be possible, or would at least be much more involved, for a line-search method formally designed on $\mM_k$ for the reasons mentioned below. Once the existence of a limit is established we may assume it to lie in $\mM_k$, in order to deduce that it is a critical point and to estimate the convergence rate. Due to the following insight (repeated as Corollary~\ref{cor: deeper insight}), such a full-rank assumption on the limit can be regarded as very natural, or even necessary in some cases, when aiming at critical-point convergence.

\textsc{Theorem.} \textit{
Let $k \le \min(m,n)$, and let $X^* \in \mM_{\le k}$ be a critical point of~\eqref{eq: min problem on bounded rank} (see section~\ref{sec: optimality conditions}). Then either $\rank(X^*) = k$ or $\nabla f(X^*) = 0$.}

Accordingly, it will be typically impossible to prove convergence to a rank-deficient critical point by a method which (in regular points) only ``sees'' projections of the gradient on tangent spaces. \changed{We therefore emphasize again that in our paper line-search methods on $\mM_{\le k}$ are \emph{not} considered for the purpose of detecting or correcting an overestimated target rank, and are also not capable of doing so.} Instead, the idea behind this work can be summarized as follows: a line-search method on $\mM_{\le k}$ \emph{can} deal with singular iterates in theory, although in the most likely cases it will not generate a single one in a real computation. Thus, in the end it will not differ from a line-search method on $\mM_k$ as used in practice, thereby establishing its convergence.

\changed{The problem of correct rank estimation remains but has been recently successfully addressed using rank-increasing algorithms~\cite{Mishraetal2013,Tanetal2014}, in which the target rank is successively increased during the process. It turns out that our concept of gradient projection on the tangent cones of $\mM_{\le k}$ is also useful in justifying and understanding such rank-increasing strategies from a theoretical perspective; see~\cite{UschmajewVandereycken2014} for further explanation.}

\subsubsection*{Contributions and outline}

The paper has two parts: in section~\ref{sec: abstract convergence theory} abstract convergence statements for line-search methods on closed sets are established. In section~\ref{sec: application to low-rank matrices} these are applied to line-search methods on $\mM_{\le k}$ with real-analytic cost function. The following list highlights the results.

\begin{itemize}
\item
Theorem~\ref{th: main theorem} states an abstract convergence result for descent methods on closed sets $\mM \subseteq \R^N$ under the assumption of a \L{}ojasiewicz inequality for the projections of the antigradient to the tangent cones. As it follows more or less known lines, the proof is provided in the appendix. 
\item
In section~\ref{eq: relation to line-search} we define line-search schemes using gradient-related search directions on tangent cones (Algorithm~\ref{A: abstract Alg}). The step-sizes are selected by backtracking to satisfy an Armijo-like rule. Our notion of a retraction (Definition~\ref{def: retraction}) is tailored to tangential projections on algebraic varieties: in every fixed tangent direction it needs to be a first-order approximation of the identity, but in contrast to a smooth retraction on a smooth manifold, it is not required to have a uniform bound on the second-order terms. The main result is Corollary~\ref{cor: smooth line-search result}: if $f$ is real-analytic, then any cluster point of the sequence generated by the method, in whose neighborhood $\mM$ forms a real-analytic submanifold, must be its limit, and a critical point of the problem.
\item
Section~\ref{sec: tangent cone of low-rank matrices} is devoted to the tangent cones of $\mM_{\le k}$. We give a much shorter derivation of their structure (Theorem~\ref{prop: form of tangent cone}) compared to~\cite{CasonAbsilVanDooren2013}. The projection on the tangent cone is a simple and feasible operation (Algorithm~\ref{Alg: Calculate gradient}). When $\rank(X) < k$, the norm of this projection can be estimated from below by the norm of the antigradient itself (Corollary~\ref{cor: projection of negative gradient on tangent cone}). This implies the above a priori statement on the rank of critical points (Corollary~\ref{cor: deeper insight}).
\item
Finally, in sections~\ref{sec: lojasiewicz for matrices} and~\ref{sec: a method without retraction} we consider two concrete line-search methods on $\mM_{\le k}$: the classical steepest descent method with projection (Algorithm~\ref{Alg: SD for matrix manifold}) and a retraction-free method using search directions which do not leave $\mM_{\le k}$ (Algorithm~\ref{Alg: without retraction}). If $f$ is real-analytic, pointwise convergence for both methods is guaranteed, but only when the limit has full rank $k$ can one conclude that it is a critical point; see Theorems~\ref{th: result for low-rank steepest descent} and~\ref{th: result for method without retraction}. We compare both algorithms for a toy example of matrix completion.
\end{itemize}

Currently, results are restricted to finite-dimensional spaces, and one has to expect that hidden constants and provable convergence rates deteriorate with the problem size. This is to be expected from a black-box tool like the \L{}ojasiewicz inequality which cannot be easily extended to infinite dimension; cf.~\cite{HarauxJendoubi2011,Huang2006}. The limitation to finite dimension has been disregarded in related works as well~\cite{CancesEhrlacherLelievre2014,Levitt2012}. On the other hand, Vandereycken~\cite{Vandereycken2013}, for instance, observed more or less dimension-independent convergence rates for matrix completion of synthetic data using nonlinear CG.

\section{Convergence of gradient methods via \L{}ojasiewicz inequality}\label{sec: abstract convergence theory}

Let $\domain \subseteq \R^N$ be open, and $f \vcentcolon \domain \to \mathbb{R}$. Throughout the paper--, unless something else is stated, we assume at least that
\begin{equation}\label{A0}
\text{$f$ is continuously differentiable and bounded below.}\tag{\textbf{A0}} 
\end{equation}
Together with $f$ we consider the minimization problem
\begin{equation}\label{eq : min problem}
\min_{x \in \mM} f(x)
\end{equation}
on a \emph{closed} subset $\mM \subset \domain$ and assume it to have a solution.  By $\| \cdot \|$ we denote the usual Euclidean norm on $\R^N$. 

\subsection{Optimality condition}\label{sec: optimality conditions}

We recall the necessary first-order optimality conditions for problem~\eqref{eq : min problem} and introduce some further notation.

Let $x \in \mM$. The \emph{tangent cone} (also called \emph{contingent cone}) at $x$ is 
\begin{equation}\label{eq: definition of tangent cone}
T_x \mM = \{ \xi \in \R^N \vcentcolon \exists (x_n) \subseteq \mM, \ (a_n) \subseteq \R^+ \text{ s.t. } x_n \to x, \  a_n(x_n - x) \to \xi  \};
\end{equation}
see, e.g.,~\cite{Guignard1969,RockafellarWets1998}. It is a closed cone. Since it is in general not convex, a metric projection onto $T_x \mM$ may not be uniquely defined. However, if we let $y \in \R^N$, then any $z \in T_{x} \mM$ with $\| y - z \| = \dist_{\| \cdot \|}(y, T_x \mM)$ is an orthogonal projection in the sense that
\begin{equation}\label{eq: Pythagoras on cone}
\| z \|^2 = \|y\|^2 - \| y - z \|^2 = \|y\|^2 - \dist_{\| \cdot \|}(y,T_x \mM)^2.
\end{equation}
Specifically, the norm of any such projection of the antigradient $- \nabla f(x)$ onto $T_x \mM$ will be denoted by
\[
g^-(x) = \sqrt{\| \nabla f(x) \|^2 - \dist_{\| \cdot \|}(-\nabla f(x),T_x \mM)^2}.
\]
An equivalent characterization, which resembles the norm of a restricted linear operator, is
\begin{equation}\label{eq: max characterization of g-}
g^-(x) = \max_{\substack{\xi \in T_x \mM \\ \| \xi \| \le 1 }} - \nabla f(x)^\trans \xi,
\end{equation}
and the maximum is achieved if and only if $\xi$ is a best approximation of $-\nabla f(x)$ in $T_x\mM$, which then must have norm $\|\xi\| = g^{-}(x)$.

The \emph{polar tangent cone}
\[
T^\circ_x \mM = \{ y \in \R^N \vcentcolon y^\trans \xi \le 0 \text{ for all $\xi \in T_x \mM$} \} 
\]
is always a closed convex cone. It equals the cone $\hat{N}_x \mM$ of regular normal vectors at $x$~\cite[Definition 6.3 and Proposition 6.5]{RockafellarWets1998}. The necessary first-order optimality condition for $x^*$ to be a relative local minimum of $f$ on $\mM$ is (see~\cite[Theorem~1]{Guignard1969} or~\cite[Theorem~6.12]{RockafellarWets1998})
\begin{equation}\label{eq: optimality condition}
-\nabla f(x^*) \in T^\circ_{x^*} \mM = \hat{N}_x \mM.
\end{equation}
Points with this property are called \emph{critical}. By~\eqref{eq: max characterization of g-}, $x^*$ is critical if and only if 
\[
g^-(x^*) = 0.
\]
This is the optimality condition we shall use in this paper.

In the case that $T_{x} \mM$ is a linear space, $T^\circ_{x} \mM$ is its orthogonal complement, $g^-(x)$ is the norm of the orthogonal projection of $\nabla f(x)$, and everything that has been said becomes quite evident. Moreover, if $\mM$ is a differentiable manifold and $\nabla f$ continuous, then $g^-$ is continuous on $\mM$. In general, it is not.

\changed{
\subsection{\L{}ojasiewicz inequality}

Our convergence results for line-search methods fundamentally rely on assuming the following property at a cluster point.

\begin{definition}
We say that $x \in \mM$ satisfies a \emph{\L{}ojasiewicz inequality for the projected antigradient} if there exist $\delta >0 $, $\Lambda >0$, and 
 $\theta \in (0, 1/2 ]$ such that for all
$ y \in \mM $ with 
$ \| y - x \| <  \delta $ it holds that
\begin{equation}\label{L}
|   f(y) - f(x) |^{1- \theta}  \leq \Lambda  g^-(y).\tag{\textbf{L}}
\end{equation}
\end{definition}

As shown in the original work by \L{}ojasiewicz~\cite[p.~92]{Lojasiewicz1965},\footnote{In this original reference the statement is that there exists $\theta \in (0,1)$ satisfying the inequality, but the proof shows that $\theta \in (0,1/2]$.} the classical, unconstrained \L{}ojasiewicz gradient inequality
\begin{equation}\label{eq:unconstrained Lojasiewicz}
|f(y) - f(x) |^{1- \theta}  \leq \Lambda  \| \nabla f(y) \|
\end{equation}
holds for the important class of real-analytic functions $f$. For this class, we can prove~\eqref{L} in the case that $\mM$ is locally the image of a real-analytic map (parametrization), for instance, as it is the case for the set $\mM_{\le k}$ (see Theorem~\ref{th: Lojasiewicz for matrices}). Basically, we need only apply the chain rule.

\begin{proposition}\label{prop: Lojasiewicz for real-analytic manifolds}
Let $f$ be real-analytic, $\mM \subseteq \domain$, and $x \in \mM$. Assume there exists $M > 0$, an open set $\mN \subseteq \R^M$, $t_0 \in \mN$, and a (componentwise) real-analytic map $\tau \vcentcolon \mN \to \R^N$ such that
\begin{itemize}
\item[\upshape (i)]
$\tau(\mN) \subseteq \mM$, $x = \tau(t_0)$, and 
\item[\upshape (ii)]
the image of every open neighborhood of $t_0$ (within $\mN$) under $\tau$ contains a relatively open neighborhood of $x$ within $\mM$ (in the induced topology of $\R^N$). 
\end{itemize}
Then~\eqref{L} holds at $x$. 
\end{proposition}

{\em Remark}. A special case arises when $\mM$ is, at least in a neighborhood of $x$, an \emph{$M$-dimensional real-analytic submanifold of $\R^N$}. In the terminology used in~\cite[Definition 2.7.1]{KrantzParks2002} this means that there exist such a map $\tau$ as in (ii) mapping \emph{any} open subset of $\mN$ \emph{onto} a relatively open subset of $\mM$ and having a derivative of rank $M$ in any point.

\begin{proof}
The composition $f \circ \tau$ is real-analytic on $\mN$~\cite[Proposition~2.2.8]{KrantzParks2002} and therefore satisfies  the classical \L{}ojasiewicz gradient inequality~\eqref{eq:unconstrained Lojasiewicz} in some open neighborhood $\widetilde{\mN} \subseteq \mN$ of $t_0$, that is,
\begin{equation}\label{eq:classical Loajsiewicz for composition}
|f(\tau(t)) - f(\tau(t_0))|^{1-\theta} \le \widetilde{\Lambda} \| \nabla(f \circ \tau)(t) \| \le \widetilde{\Lambda} \| \nabla \tau(t) \cdot \nabla f(\tau(t)) \|
\end{equation}
for all $t \in \widetilde{\mN}$. As $\tau$ maps on $\mM$, it is easy to show that the derivative $\tau'(t) = [\nabla \tau(t)]^\trans$ maps onto the tangent cone $T_{\tau(t)} \mM$ at $\tau(t)$. Using~\eqref{eq: max characterization of g-}, we deduce
\begin{equation}\label{eq:estimate for chain derivative}
\begin{aligned}
\| \nabla \tau(t) \cdot \nabla f(\tau(t)) \| &= \max_{\| h \|=1} - \nabla f(\tau(t))^\trans \tau'(t)h \\
&\le \max_{\| h \|=1} g^-(\tau(t)) \| \tau'(t)h \| = g^-(\tau(t)) \| \tau'(t) \|.
\end{aligned}
\end{equation}
Without loss of generality, we can assume that $\| \tau'(t) \| \le C$ for all $t \in \widetilde{\mN}$. Combining~\eqref{eq:classical Loajsiewicz for composition} and~\eqref{eq:estimate for chain derivative} then proves~\eqref{L} for $\Lambda = \widetilde{\Lambda} C$, since by (ii) there exists a $\delta > 0$ such that every $y \in \mM$ with $\| y- x \| < \delta$ can be written as $y = \tau(t)$ with $t \in \widetilde{\mN}$.
\qquad\end{proof}

The actual values of $\theta$ and $\Lambda$ may depend on $x$ and are typically not known. We will see below that the strongest convergence statements are achieved when $\theta = \frac{1}{2}$. The two generic cases in which this happens in the unconstrained version~\eqref{eq:unconstrained Lojasiewicz} (for $\delta$ small enough) are $\nabla f(x) \neq 0$ or $\nabla f(x) = 0$ with positive definite Hessian. Of these two cases, only the second is of interest, and we could make a similar statement in the setting of Proposition~\ref{prop: Lojasiewicz for real-analytic manifolds} by assuming the Hessian of $f \circ \tau$ to be positive definite at $t_0$. However, this does not seem to be very useful. First, it is not clear how such an assumption could be related to more concrete conditions on $f$ and $\mM$. Second, in the context of line-search methods we shall consider below, the case of a 
positive definite Hessian at a cluster point could likely be treated by a more constructive local convergence analysis. In summary, the value of $\theta$ will remain unknown in our subsequent results.
}

\subsection{General convergence theorem}\label{sec: global convergence theorem}

Here we state a meta convergence theorem. Consider some iteration $(x_n) \subseteq \mM$ that is intended to solve~\eqref{eq : min problem}. Throughout the paper we will use the shorthand
\[
f_n = f(x_n), \quad \nabla f_n = \nabla f (x_n), \quad g^-_n = g^-( x_n), \quad \text{and} \quad T_n \mM = T_{x_n} \mM.
\]
Using this notation, we make the following assumptions.

\begin{itemize}[leftmargin=*]
\item
\textbf{Primary descent condition:} There exists $\sigma >0 $ such that for large enough $n$ it holds that
\begin{equation}\label{A1}
f_{n+1} - f_n \leq -\sigma  g^-_n \| x_{n+1} - x_n \|.\tag{\textbf{A1}} 
\end{equation}
\item
\textbf{Stationary condition:} For large enough $n$ it holds that
\begin{equation}\label{A2} 
g^-_n = 0 \quad \Rightarrow \quad x_{n+1} = x_n.\tag{\textbf{A2}}
\end{equation}
\item
\textbf{Asymptotic small step-size safeguard:} There exists $\kappa >0 $ such that for large enough $n$ it holds that
\begin{equation}\label{A3}
\| x_{n+1}  -  x_n \|   \geq  \kappa g^-_n.\tag{\textbf{A3}}
\end{equation}
\end{itemize}

In combination with a \L{}ojasiewicz inequality~\eqref{L}, these assumptions imply a fairly strong convergence result.
\begin{theorem}\label{th: main theorem}
Under assumptions {\upshape\eqref{A1}--\eqref{A2}}, if there exists a cluster point $x^*$ of the sequence $(x_n)$ satisfying~\eqref{L}, it is actually its limit point. Further if~\eqref{A3} holds, then the convergence rate can be estimated by
\[
\| x_n - x^* \| \lesssim \begin{cases} e^{- cn } \quad &\text{if $\theta = \frac{1}{2}$ (for some $c>0$)}, \\ n^{- \frac{\theta }{1 -2 \theta }} \quad &\text{if $0<\theta<\frac{1}{2}$.}  \end{cases}
\]
Moreover, $g^-_n \to 0$.
\end{theorem}

This theorem is an adaption of similar results scattered throughout the literature. Up to replacing the usual gradient by the projected antigradient, assumptions~\eqref{A1},~\eqref{A2} and~\eqref{L} are the same as in~\cite[Theorem~3.2]{AbsilMahonyAndrews2005} and are sufficient to prove the convergence of the sequence $(x_n)$ if it is bounded. \eqref{A2} is a natural technical requirement to the algorithm for not moving in the critical-point set and is typically satisfied if the iteration is gradient-related. Adding assumption~\eqref{A3} does not only guarantee that the $g_n^-$ tend to zero, but it allows us to estimate the convergence rate along known lines, e.g.,~\cite{AttouchBolte2009,Levitt2012}. However, as~\eqref{A3} is required here only for $n$ larger than some unknown $n_0$, one cannot determine the constants behind $\lesssim$ explicitly (a constant depending on $n_0$ may be deduced). 

Corresponding results for smooth manifolds have been obtained in~\cite{Lageman2007a,Lageman2007,MerletNguyen2013}. In this context, we should remark that the ambient norm $\| x_{n+1} - x_n \|$, as used in~\eqref{A1} and~\eqref{A3}, is not necessarily a natural measure of distance on $\mM$, but is particularly appropriate when the restriction to $\mM$ is motivated to reduce the complexity of a minimization problem in $\R^N$, as is typically the case for low-rank optimization.\

Although no changes in the known arguments besides replacing $\|\nabla f\|$ by $g^-$ are required, we give a proof of Theorem~\ref{th: main theorem} in the appendix to keep the paper self-contained.

\changed{
We emphasize that the theorem only states $g_n^- \to 0$.} The question of when this actually implies $g^-(x^*) = 0$ is delicate, and simple counterexamples can be constructed. A sufficient condition would be $T_{x^*} \mM \subset \liminf_{n \to \infty} T_n \mM$ in the sense of set convergence (see, e.g.,~\cite{RockafellarWets1998}). Unfortunately, this will usually not hold in the singular points of $\mM_{\le k}$ when approached by a sequence of full-rank matrices (Theorem~\ref{prop: form of tangent cone}). Later, we will be forced to make some smoothness assumptions on a neighborhood of $x^*$.

\subsection{Retracted line-search methods}\label{eq: relation to line-search}

For line-search methods in $\R^N$ it is well known how to obtain convergence results based on the \L{}ojasiewicz gradient inequality~\cite{AbsilMahonyAndrews2005}. Here we consider projected gradient flows on a set $\mM$. 

\subsubsection{Retractions}

Following~\cite{AbsilMahonySepulchre2008}, a retracted line-search method on a smooth manifold $\mM$ has the general form
\begin{equation}\label{eq: line-search scheme}
x_0 \in \mM, \quad x_{n+1} = R(x_n, \alpha_n \xi_n),
\end{equation}
where $\xi_n$ are tangent vectors at $x_n$, $\alpha_n \ge 0$, and
\(
R \vcentcolon T\mM  \to \mM
\)
is a \emph{smooth retraction}~\cite{Shub1986}. This means that $R$ is a $C^\infty$ map which takes pairs $(x, \xi_x)$ from the tangent bundle $T \mM$ (which represent vectors $x + \xi_x$ on the affine tangent plane at $x$) back to the manifold, and has the property of being a first-order approximation of the exponential map, that is, its derivative at $(x,0)$ with respect to $\xi_x$ is the identity on $T_x \mM$:
\begin{equation}\label{eq: first order approximation}
	\lim_{T_{x} \mM  \ni \xi_x \to 0}\frac{\| R(x, \xi_x) - (x + \xi_x) \|}{\| \xi_x \|} = 0
\end{equation}
for all $x \in \mM$. However, since we do not want to restrict ourselves to smooth manifolds, we make the following, more general definition.
\begin{definition}[retraction]\label{def: retraction}
Let $\mM$ be closed. A map
\[
R \vcentcolon \bigcup_{x \in \mM} \{ x \} \times T_x\mM  \to \mM
\]
(where now $T_x \mM$ is the tangent cone) will be called a \emph{retraction} if for any fixed $x \in \mM$ and $\xi_x \in T_x \mM$ it holds that $\alpha \mapsto R(x, \alpha \xi_x)$ is continuous on $[0,\infty)$, and
\begin{equation}\label{eq: Retraction o of alpha}
\lim_{\alpha \to 0^+} \frac{R(x, \alpha \xi_x) - (x + \alpha \xi_x)}{\alpha} = 0.
\end{equation}
\end{definition}
The existence of such a retraction has implications for the regularity of the set $\mM$. It is equivalent to the (one-sided) differentiability of the map $\alpha \mapsto \dist_{\| \cdot \|}(x + \alpha \xi_x, \mM)$ in zero. This is, for instance, the case for real-algebraic varieties like $\mM_{\le k}$, and follows from the fact that for every tangent vector $\xi_x$ to an algebraic variety, there exists an analytic arc $\gamma \vcentcolon [0,\epsilon) \to \mM$ such that $\xi_x = \dot{\gamma}(0)$~\cite[Proposition~2]{OSheaWilson2004}.

By~\eqref{eq: Retraction o of alpha}, $R(x + \alpha\xi_x)$ is better than a first-order approximation of $x + \alpha\xi_x$ for very small $\alpha > 0$. In particular, for any fixed $\xi_x$ and $\epsilon > 0$,~\eqref{eq: Retraction o of alpha} implies that
\begin{equation}\label{eq: asymptotic approximation of stepsize}
(1 - \epsilon)\alpha \| \xi_x\| \le \| R(x, \alpha\xi_x) - x \| \le (1 + \epsilon)\alpha \|\xi_x\| \quad \text{for sufficiently small $\alpha$.}
\end{equation}
It means that a (small enough) step made in the tangent cone is neither increased nor decreased too much by the retraction, which obviously is of importance in analyzing a line-search method like~\eqref{eq: line-search scheme}. In what follows, we assume that we have a general upper bound for arbitrary steps:
\begin{equation}\label{eq: upper bound for retraction}
\| R(x, \xi_x) - x \| \le M \| \xi_x \| \quad \text{for all $x \in \mM$ and $\xi_x \in T_x \mM$.}
\end{equation}
This imposes no serious restriction.

Since $\mM$ is assumed to be closed, a natural choice for $R$, though practically not always the most convenient, is the best approximation of $x + \xi_x$ in the Euclidean ambient norm (metric projection), that is,
\begin{equation}\label{eq: best approximation}
R(x, \xi_x) \in \argmin_{y \in \mM} \| y - (x + \xi_x) \|.
\end{equation}
By the remarks above, this defines a valid retraction, for example, on closed real-algebraic varieties (cf.~\eqref{eq: valid retraction on variety}) with $M = 2$ in~\eqref{eq: upper bound for retraction}. For the variety $\mM_{\le k}$ of bounded rank matrices one even can take $M = 1 + 2^{-1/2}$ (Proposition~\ref{prop: stable retraction}).

\subsubsection{Angle condition}

To obtain such strong convergence results as we have in mind, one naturally has to guarantee that the search directions $\xi_n$ in~\eqref{eq: line-search scheme} remain sufficiently gradient-related. We call $\xi_n \in T_n \mM$ a \emph{descent direction} if $\nabla f_n^\trans\xi_n < 0$.

\begin{definition}[angle condition]
Given  $x_n \in \mM$ and $\omega \in (0,1]$, $\xi_n \in T_n \mM$ is said to satisfy the \emph{$\omega$-angle condition} if
\begin{equation}\label{eq: angle condition}
\nabla f_n^\trans \xi_n^{} \le - \omega g_n^- \| \xi_n \|.
\end{equation}
\end{definition}

An equivalent statement is that the inner product between $-\nabla f_n / \| \nabla f_n\|$ and $\xi_n / \| \xi_n \|$ is at least $\omega g_n^- / \| \nabla f_n \|$. 

For clarity, we emphasize the following.

\begin{proposition}\label{prop: existence of decrease direction}
Any Euclidean best approximation
\[
\xi_n \in \argmin_{\xi \in T_n \mM} \| - \nabla f_n - \xi \|
\]
of $- \nabla f_n$ on $T_n \mM$ satisfies the $\omega$-angle condition with $\omega = 1$. Moreover, with this choice, $\xi_n = 0$ if and only if $g_n^- = 0$.
\end{proposition}
\begin{proof}
As discussed in section~\ref{sec: optimality conditions}, it holds in this case that $g^-_n = \| \xi_n^{} \| = \sqrt{\| \nabla f_n^{} \|^2 - \| \nabla f_n^{} + \xi_n^{} \|^2}$, which implies $\nabla f_n^\trans \xi_n^{} = - g_n^- \| \xi_n^{} \|$.
\qquad\end{proof}

\subsubsection{Armijo point}

Given $x_n \in \mM$ and a descent direction $\xi_n \in T_n \mM$, we will have to pick a step-size $\alpha_n$ small enough to satisfy~\eqref{A1}. It should, however, be as large as possible in order to hopefully guarantee~\eqref{A3}.

\begin{definition}[Armijo point]\label{def: Armijo point}
Let $\xi_n \in T_n \mM$ be a descent direction at $x_n \in \mM$, $\bbeta_n >0$, and $\beta,c \in (0,1)$. The number
\begin{equation}\label{eq:Armijo point}
\alpha_n = \max\{ \beta^m\bbeta_n \vcentcolon m \in \N \cup \{0\}, \ f(R(x_n, \beta^m \bbeta_n \xi_n)) - f_n \le c \beta^m \bbeta_n \nabla f_n^\trans \xi_n^{} \}
\end{equation}
is called the \emph{Armijo point} for $x_n,\xi_n,\bbeta_n,\beta,c$.
\end{definition}

This will be our choice for the step-size $\alpha_n$ in all subsequent algorithms. The importance of the Armijo point lies in the fact that in principle it can be found in finitely many steps using backtracking. To see that the maximum in~\eqref{eq:Armijo point} is not taken over the empty set, we introduce another important point:
\begin{equation}\label{eq: definition alpha prime}
\balpha_n = \min\{ \alpha > 0 \vcentcolon f(R(x_n, \alpha \xi_n)) - f_n = c \alpha \nabla f_n^\trans \xi_n^{} \}.
\end{equation}
Then the following relations hold.

\begin{proposition}\label{prop: propoerties of Armijo point}
Assume~\eqref{A0}. Let $\xi_n \in T_n \mM$ be a descent direction at $x_n \in \mM$, and $\beta,c \in (0,1)$. Then $\balpha_n > 0$ exists, i.e., the minimum in~\eqref{eq: definition alpha prime} is not taken over the empty set, and $f(R(x_n, \alpha \xi_n)) - f_n \le c \alpha \nabla f_n^\trans \xi_n^{}$ for all $\alpha \in [0,\balpha_n]$. The Armijo point $\alpha_n$ defined by~\eqref{eq:Armijo point} satisfies
 \begin{align*}\alpha_n \ge \beta \balpha_n &\quad \text{if $\bbeta_n > \balpha_n$,} \\ \alpha_n = \bbeta_n &\quad \text{if $\bbeta_n \le \balpha_n$.}
\end{align*}
\end{proposition}
\begin{proof}
For convenience, let $\hR(\alpha) = R(x_n, \alpha \xi_n)$ and $F(\alpha) = f_n + c \alpha \nabla f_n^\trans \xi_n^{}$. We have to show that $f(\hR(\alpha))$ is strictly smaller than $F(\alpha)$ for sufficiently small $\alpha > 0$. By Taylor's theorem and~\eqref{eq: Retraction o of alpha},
\begin{align*}
f(\hR(\alpha)) 
&= f_n + \nabla f_n^\trans (\hR(\alpha) - x_n) + o(\| \hR(\alpha) - x_n\|)\\
&= f_n+ \nabla f_n^\trans (\alpha \xi_n + o(\alpha)) + o(\| \hR(\alpha) - x_n\|)\\
&= f_n+ c\alpha \nabla f_n^\trans \xi_n^{} + (1-c) \alpha f_n^\trans \xi_n^{} + o(\alpha) + o(\| \hR(\alpha) - x_n\|),
\end{align*}
where $o(h)$ denotes a quantity with $o(h)/h = 0$ for $h \to 0^+$. By~\eqref{eq: asymptotic approximation of stepsize}, the ratio $\| \hR(\alpha) - x_n\|/\alpha$ converges to $\|\xi_n\|$ for $\alpha \to 0^+$, which implies $o(\| \hR(\alpha) - x_n\|) = o(\alpha)$. As desired, it now follows that
\[
\frac{1}{\alpha}(f(\hR(\alpha)) - f_n - c\alpha \nabla f_n^\trans \xi_n^{}) = (1-c) \nabla f_n^\trans \xi_n^{} + \frac{o(\alpha)}{\alpha}
\]
is negative for small enough $\alpha$. Since $\alpha \mapsto \hR(\alpha)$ is continuous and bounded below by~\eqref{A0}, while $F(\alpha)$ is not, the smallest positive intersection point $\balpha_n$ must exist. The assertions on $\alpha_n$ are immediate.
\qquad\end{proof}

The role of the parameter $\bbeta_n$ in Definition~\ref{def: Armijo point} is to adjust the initial length of the search direction $\xi_n$ on which no assumptions have been made. When $\bbeta_n \|\xi_n\|$ is too small, one has no chance to establish a minimum step-size safeguard like~\eqref{A3}. The restriction we make is
\begin{equation}\label{eq: condition for gamma_n}
\bbeta_n \ge \min\left( \frac{g_n^-}{\| \xi_n \|}, \balpha_n \right).
\end{equation}
To achieve this, one needs to either calculate $g_n^-$, or increase the value of $\bbeta_n$ until $f(R(x_n, \bbeta \xi_n)) \ge f_n + c \bbeta_n \nabla f_n^\trans \xi_n^{}$ holds.

\subsubsection{Convergence results} 

The algorithm we analyze is formalized as Algorithm~\ref{A: abstract Alg}. By Propositions~\ref{prop: existence of decrease direction} and~\ref{prop: propoerties of Armijo point}, all steps are feasible. We first assert that the mere convergence of the produced sequence $(x_n)$, when assuming the \L{}ojasiewicz inequality, is guaranteed by Theorem~\ref{th: main theorem}. 

\begin{algorithm}
\KwIn{Starting point $x_0 \in \mM$, $\omega \in (0,1]$, $\beta \in (0,1)$, $c \in (0,1)$.}
\For{n=0,1,2,\dots}{
Choose $\xi_n \in T_n \mM$ satisfying~\eqref{eq: angle condition}, but choose $\xi_n = 0$ only when $g_n^- = 0$\;
Choose $\bbeta_n \ge \min(g_n^- / \| \xi_n \| \balpha_n )$; find Armijo point $\alpha_n$ for $x_n, \xi_n, \bbeta_n, \beta, c$\;
Form the next iterate
\[
x_{n+1} = R(x_n, \alpha_n \xi_n).
\]
}
\caption{Gradient-related projection method with line-search}\label{A: abstract Alg}
\end{algorithm}

\begin{corollary}\label{cor: mere convergence of line-search}
Assume~\eqref{A0}. The sequence $(x_n)$ produced by Algorithm~\ref{A: abstract Alg} satisfies~\eqref{A1} with $\sigma = \omega c M^{-1}$ ($M$ being the constant from~\eqref{eq: upper bound for retraction}) and~\eqref{A2}. Consequently, if a cluster point $x^*$ exists and satisfies the \L{}ojasiewicz gradient inequality~\eqref{L}, then $\lim_{n \to \infty} x_n = x^*$.
\end{corollary}
\begin{proof}
Property~\eqref{A1} follows immediately from~\eqref{eq: angle condition} and~\eqref{eq: upper bound for retraction};~\eqref{A2} holds by construction.
\qquad\end{proof}

Obviously, it is not necessary to choose the Armijo step-size to obtain this result; it suffices to have $f(R(x_n, \alpha_n \xi_n)) - f_n \le c \alpha_n \nabla f_n^\trans \xi_n^{}$. The choice of the Armijo point becomes important, however, when one also aims for~\eqref{A3} and the convergence rate estimate in Theorem~\ref{th: main theorem}. To proceed in this direction, we were not able to avoid imposing additional regularity assumptions on the retraction in the limit point.

\begin{theorem}\label{th: abstract theorem for line-search}
In the situation of Corollary~\ref{cor: mere convergence of line-search}, assume further that
\begin{itemize}
 \item[\upshape (i)]
 $\alpha_n \xi_n \to 0$, and
 \item[\upshape (ii)]
 there exists a constant $C > 0$ such that for all sequences $(\hat{\xi}_n)$ with $\hat{\xi}_n \in T_n \mM$ and $\hat{\xi}_n \to 0$ it holds that
\begin{equation}\label{eq: abstract property}
 \limsup_{n \to \infty} \frac{\| R(x_n, \hat{\xi}_n) - (x_n + \hat{\xi}_n)\|}{\| \hat{\xi}_n \|^2} \le C.
\end{equation}
\end{itemize}
Assume further that $f$ is bounded below on the whole of $\R^N$, and that there exists an open (in $\R^N$) neighborhood $\mN$ of $x^*$ such that
\begin{equation}\label{eq: Lipschitz condition}
\| \nabla f(x) - \nabla f(y) \| \le L \| x - y \| \quad \text{for all $x,y \in \mN$.}
\end{equation}
Then~\eqref{A3} holds (with a generally unknown constant $\kappa$). Consequently, $g_n^- \to 0$, and the convergence rate estimates in Theorem~\ref{th: main theorem} apply.
\end{theorem}

We discuss these two conditions after the proof.

\begin{proof}
We can assume $g_n^- > 0$ for all $n$, since otherwise the sequence becomes stationary. Then we have $\| \alpha_n \xi_n \| > 0$ for all $n$. We have to show that $\liminf_{n \to \infty} \| x_{n+1} - x_n \| / g_n^- > 0$. We do this by showing that the assumption $\liminf_{n \to \infty} \| x_{n+1} - x_n \| / g_n^- = 0$ leads to a contradiction. In the following we consider a subsequence which converges to the limes inferior, but for notational convenience we assume that 
\begin{equation}\label{eq: liminf assumption}
\lim_{n \to \infty} \frac{\| x_{n+1} - x_n \|}{g_n^-} = 0.
\end{equation}

Fix $m \in (0,1)$. As $\alpha_n \xi_n \to 0$,~\eqref{eq: abstract property} implies that for large enough $n$ we will have
\[
\| \alpha_n \xi_n \| \le \| x_{n+1} - x_n - \alpha_n \xi_n \| + \| x_{n+1} -x_n \| \le C \| \alpha_n \xi_n \|^2 + \| x_{n+1} - x_n \|.
\]
We consider $n$ so large that $m \le (1 - C \|\alpha_n \xi_n \|)$ or, after rearranging,
\begin{equation}\label{eq: asymptotic lower bound for retraction}
 m\| \alpha_n \xi_n \| \le \| x_{n+1} - x_n\|.
\end{equation}
Since $\mN$ is open and $x_n \to x^* \in \mN$ and $\hat{\xi}_n \to 0$, it also holds that
\begin{equation}\label{eq: z in Lipschitz domain}
x_n + z \in \mN \quad \text{for all $z$ with $\| z \| \le M \beta^{-1} \| \hat{\xi}_n \|$}
\end{equation}
if only $n$ is large enough. Hence we may assume, without loss of generality, that~\eqref{eq: asymptotic lower bound for retraction} and~\eqref{eq: z in Lipschitz domain} hold for all $n$. Now we distinguish the iterates by two disjoint cases: $\bbeta_n \le \balpha_n$ and $\bbeta_n > \balpha_n$. In the first case, we have $\alpha_n = \bbeta_n$ by Proposition~\ref{prop: propoerties of Armijo point}, 
which by the choice of $\bbeta_n$ in the algorithm according to~\eqref{eq: condition for gamma_n} gives
\[
\| x_{n+1} - x_n \| \ge m \| \alpha_n \xi_n \| \ge m g_n^-.
\]
Assumption~\eqref{eq: liminf assumption} implies that this happens only for finitely many $n$. Let us therefore assume that the second case $\bbeta_n > \balpha_n$ always occurs. In this case, Proposition~\ref{prop: propoerties of Armijo point} states that $\balpha_n \le \beta^{-1} \alpha_n$. Hence, by~\eqref{eq: asymptotic lower bound for retraction} and~\eqref{eq: liminf assumption},
\begin{equation}\label{eq: liminf for alpha prime}
\lim_{n \to \infty} \frac{\| \balpha_n \xi_n \|}{g_n^-} \le \lim_{n \to \infty} \frac{m^{-1} \beta^{-1}   \| x_{n+1} - x_n \|}{g_n^-} = 0.
\end{equation}

We now show that~\eqref{eq: liminf for alpha prime} leads to a contradiction by mimicking arguments that are used to prove existence of step-sizes satisfying the strong Wolfe conditions in linear spaces, e.g.,~\cite[Lemma~3.1]{NocedalWright2006}. Let again $\hR(\alpha) = R(x_n, \alpha \xi_n)$. By the mean value theorem, there exists $\vartheta \in (0,1)$ such that $z = \vartheta(\hR(\balpha_n) - x_n)$ satisfies
\begin{equation}\label{eq: mean value theorem}
(\hR(\balpha_n) - x_n)^\trans \nabla f(x_n + z) = f(\hR(\balpha_n)) - f_n = c \balpha_n \xi_n^\trans \nabla f_n,
\end{equation}
where the second equality holds by definition~\eqref{eq: definition alpha prime}. By~\eqref{eq: upper bound for retraction}, $\| z \| \le M\|\balpha_n \xi_n \| \le M\beta^{-1}\|\alpha_n \xi_n\|$ so that $x_n + z \in \mN$ by~\eqref{eq: z in Lipschitz domain}. Using~\eqref{eq: Lipschitz condition}, Cauchy--Schwarz, the definition of $z$, the reverse triangle inequality, and the angle condition~\eqref{eq: angle condition}, we can estimate:
\begin{equation*}\label{eq: big estimate}
\begin{aligned}
\| z \| \| \hR(\balpha_n) - x_n \| &\ge L^{-1} \| \nabla f_n - \nabla f(x_n + z) \| \| \hR(\balpha_n) - x_n \| \\
&\ge L^{-1}\lvert \nabla f_n^\trans(\hR(\balpha_n) - x_n) - c \balpha_n \nabla f_n^\trans \xi_n^{} \rvert \\
&\ge L^{-1} ((1-c) \omega g_n^- \| \balpha_n \xi_n \| - \lvert \nabla f_n^\trans(\hR(\balpha_n) - (x_n + \balpha_n \xi_n)) \rvert).
\end{aligned}
\end{equation*}
Since we have $\| \balpha_n \xi_n \| \ge M^{-1} \| \hR(\balpha_n) - x_n \| \ge M^{-1} \|z\|$, we arrive at
\begin{equation}\label{eq: arrive at}
\frac{\| \balpha_n \xi_n \|}{g_n^-} \ge M^{-2} L^{-1} \left( (1-c)\omega - \frac{\lvert \nabla f_n^\trans(\hR(\balpha_n) - (x_n + \balpha_n \xi_n)) \rvert}{g_n^- \| \balpha_n \xi_n \|} \right).
\end{equation}
By assumption, $\|\balpha_{n} \xi_{n} \| \le \beta^{-1} \| \alpha_{n} \xi_{n} \| \to 0$. Since $\nabla f$ is continuous, it follows from Cauchy--Schwarz,~\eqref{eq: abstract property}, and~\eqref{eq: liminf for alpha prime} that
\[
\lim_{k \to \infty} \frac{\lvert \nabla f_n^\trans(\hR(\balpha_n) - (x_n + \balpha_n \xi_n)) \rvert}{g_n^- \| \balpha_n \xi_n \|} \le \lim_{k \to \infty} \frac{\|\nabla f(x^*)\| C \| \balpha_n \xi_n \|^2}{g_n^- \| \balpha_n \xi_n \|} = 0.
\]
Therefore,~\eqref{eq: arrive at} yields
\[
 \liminf_{n \to \infty} \frac{\| \balpha_n \xi_n \|}{g_n^-} \ge M^{-2}L^{-1}(1-c)\omega > 0,
\]
in contradiction to~\eqref{eq: liminf for alpha prime}.
\qquad\end{proof}

\changed{
Property~\eqref{eq: abstract property} in Theorem~\ref{th: abstract theorem for line-search} holds, for instance, if $\mM$ is locally a smooth submanifold in a neighborhood of the limit $x^*$ of $(x_n)$, and $R$ is a smooth retraction in that neighborhood (one can  locally bound the second derivates of $\xi_x \mapsto R(x,\xi_x)$). On the other hand, the condition $\alpha_n \xi_n \to 0$ is reasonable, but cannot be removed in general.\footnote{The reason is that our requirements on $R$ are only of a local kind. Consider, as a counterexample, the exponential retraction on a unit sphere; i.e., $R(x,\xi)$ is the endpoint of the arc of length $\|\xi_x\|$ in the great circle from $x$ in direction $\xi$. Then for $\|\xi_x\| \to 2\pi$ it holds that $x - R(x,\xi_x) \to 0$.} When $R$ is concretely specified, the situation can change. Considering the relevant example where a metric projection is used as retraction, it turns out that $x_{n+1} - x_n \to 0$ in combination with $x_n \to x^*$ automatically implies $\alpha_n \xi_n \to 0$ in the smooth case. This leads to the following powerful corollary of Theorem~\ref{th: abstract theorem for line-search}.}

\begin{corollary}\label{cor: smooth line-search result}
Let $f$ be real-analytic and bounded below. Assume the metric projection~\eqref{eq: best approximation} (the choice of norm does not matter here) is used as retraction in Algorithm~\ref{A: abstract Alg}. Further assume a cluster point $x^*$ of the sequence $(x_n)$  produced by Algorithm~\ref{A: abstract Alg} exists, satisfies the \L{}ojasiewicz gradient inequality~\eqref{L}, and possesses an open neighborhood $\mO \subseteq \R^N$ such that $\mM \cap \mO$ is a smooth embedded submanifold of $\R^N$. Then~\eqref{A1}--\eqref{A3} hold. Consequently, $\lim_{n \to \infty} x_n = x^*$ with a rate of convergence as indicated in Theorem~\ref{th: main theorem}, and $\lim_{n \to \infty} g_n^- = g^-(x^*) = 0$.
\end{corollary}

\changed{
{\em Remark}. More concisely one may assume that $\mM \cap \mO$ is a real-analytic submanifold of $\R^N$~\cite[Definition~2.7.1]{KrantzParks2002}. Then the assumption on the validity of~\eqref{L} is superfluous; cf. the remark following Proposition~\ref{prop: Lojasiewicz for real-analytic manifolds}.
}

\begin{proof}
By Corollary~\ref{cor: mere convergence of line-search}, $x_n \to x^*$ and $\lim_{n \to \infty} g_n^- = g^-(x^*)$ (since on a smooth manifold $g^-$ is a continuous function). For completeness, we now sketch the more or less elementary arguments that $\alpha_n \xi_n \to 0$ and~\eqref{eq: abstract property} hold. Then Theorem~\ref{th: abstract theorem for line-search} applies (the local Lipschitz condition for the gradient follows from the analyticity assumption).

There exists a local diffeomorphism $\phi$ from a neighborhood of $0 \in T_{x^*} \mM$ (which is a linear space now) to $\mM$ such that for large enough $n$ we can write $x^* = \phi(0)$, $x_n = \phi(y_n)$, and $T_n \mM = \ran(\phi'(y_n))$. The optimality condition for $x_{n+1} = R(x_n + \alpha_n \xi_n)$ when it is the orthogonal projection of $x_n + \alpha_n \xi_n$ is that the error is orthogonal on the tangent space at $x_{n+1}$, i.e.,
\[
0 =  \eta^\trans \phi'(y_{n+1})^\trans (x_{n+1} - (x_n + \alpha_n \xi_n) )  \quad \text{for all $\eta \in T_{x^*} \mM$.}
\]
As $x_{n+1} - x_n \to 0$, this implies
\[
0 = \lim_{n \to \infty} \alpha_n \phi'(y_{n+1})^\trans \xi_n.
\]
Since the smallest singular value of $\phi'(y_{n+1})^\trans$ can be uniformly bounded below for $n$ large enough (the limit $\phi'(0)^\trans$ has full rank), it follows that $\alpha_n \xi_n \to 0$. Further, for any $\hat{\xi}_n = \phi'(y_n)\hat{\eta}_n$ we have by the best approximation property of $R$ and Taylor's theorem that
\[
\| R(x_n + \hat{\xi}_n) - (x_n + \hat{\xi}_n) \| \le \| \phi(y_n + \hat{\eta}_n) - (\phi(y_n) + \phi'(y_n) \hat{\eta}_n ) \| \le \| \phi''(x_n + \vartheta_n \hat{\xi}_n)\| \|\hat{\eta}_n \|^2
\]
for some $\vartheta_n \in (0,1)$. If $\hat{\xi}_n \to 0$ for $n \to \infty$, then it follows that
\[
\limsup_{n \to \infty} \frac{\| R(x_n + \hat{\xi}_n) - (x_n + \hat{\xi}_n) \|}{\| \hat{\xi}_n \|^2} \le \| \phi''(0) \| \| (\phi'(0))^{-1} \|,
\]
since $\phi''$ is continuous in zero.
\qquad\end{proof}

\section{Results for matrix varieties of bounded rank}\label{sec: application to low-rank matrices}

The space $\R^m \otimes \R^n \cong \R^{m \times n} \cong \R^{mn}$ becomes a Euclidean space when equipped with the Frobenius inner product $\langle X, Y \rangle_\frob = \tr(X^\trans Y)$. The corresponding norm and distance function are denoted by $\| \cdot \|_\frob$ and $\dist_\frob$, respectively. Points in this space will now be denoted by $X$ instead of $x$, tangent vectors by $\Xi$ instead of $\xi$. Mainly to save space, we prefer in this paper the subspace and tensor product notation over explicit matrix representations. However, if we use the latter (as in the definition of the inner product), then it is with respect to some fixed orthonormal bases in $\R^m$ and $\R^n$. For example, writing $X \in \mU \otimes \mV$ in $\R^m \otimes \R^n$ would mean in $\R^{m \times n}$ that $X = U S V^\trans$ for some matrices $U,S,V$ with $\ran(U) = \mU$ and $\ran(V) = \mV$. By $\Pi_\mS$ we denote the orthogonal projection onto a subspace $\mS$. Then $(\Pi_\mU \otimes \Pi_\mV)X$ corresponds to $U U^\trans X V V^\trans$, 
where $U$ and $V$ are orthonormal basis representations for $\mU$ and $\mV$, respectively. 

In this section we apply the above convergence theory for line-search methods to the real-algebraic variety $\mM_{\le k}$ of matrices with rank at most $k$ (see~\eqref{eq: bounded rank matrices}). We consider the problem
\begin{equation}\label{eq: low-rank min problem}
\min_{X \in \mM_{\le k}} f(X),
\end{equation}
where, as before, $f \vcentcolon \R^{m \times n} \supseteq \mD \to \R$ is continuously differentiable and bounded below. In fact, in the end we will assume that $f$ is real-analytic to ensure the \L{}ojasiewicz gradient inequality.

\subsection{Tangent cone and optimality}\label{sec: tangent cone of low-rank matrices}

Here and in the following, we suppose that
\[
\rank(X) = s \le k, \quad \mU = \ran(X), \quad \mV = \ran(X^\trans).
\]
The following is well known; see, e.g.,~\cite{HelmkeShayman1995,KochLubich2007,Vandereycken2013}.

\begin{theorem}\label{th: smoothness of full-rank manifolds}
The set $\mM_s$ of rank-$s$ matrices is a smooth submanifold of dimension $(m + n - s)s$. It is dense and relatively open in $\mM_{\le s}$. The tangent space of $\mM_s$ at $X$ is
\begin{equation}\label{eq: tangent space Ms}
T_X \mM_s = (\mU \otimes \mV) \oplus (\mU^\bot \otimes \mV) \oplus (\mU \otimes \mV^\bot).
\end{equation}
The orthogonal projector on $T_X \mM_s$ is hence given by
\begin{equation}\label{eq: projector on tangent space}
\Pi_{T_X \mM_s} = \Pi_\mU \otimes \Pi_\mV + \Pi_{\mU^\bot} \otimes \Pi_\mV + \Pi_\mU \otimes \Pi_{\mV^\bot} = \Pi_\mU \otimes I + I \otimes \Pi_\mV - \Pi_\mU \otimes \Pi_\mV,
\end{equation}
and it holds that
\begin{equation}\label{eq: orthogonal decomposition}
T_X \mM_s \oplus (\mU^\bot \otimes \mV^\bot) = \R^m \otimes \R^n.
\end{equation}
\end{theorem}

\changed{
In fact, $\mM_s$ is even a real-analytic submanifold; see Lemma~\ref{lem: M_s real-analytic submanifold}.
}

Our main task is to investigate the tangent cones of $\mM_{\le k}$ in points with $s < k$. The tangent cone $T_X \mM_{\le k}$ clearly contains $T_X \mM_s$, but, in case $s< k$, also contains rays that arise when approaching $X$ by a matrix of rank at most $k$ but larger than $s$.

\begin{theorem}{\upshape(see~\cite{CasonAbsilVanDooren2013})}\label{prop: form of tangent cone}
Let $X \in \mM_{\le k}$, $\rank(X) = s$. The tangent cone of $\mM_{\le k}$ at $X$ is
\[
T_X \mM_{\le k} = T_X \mM_s \oplus \{ \Xi_{k-s} \in \mU^\bot \otimes \mV^\bot \vcentcolon \rank(\Xi_{k-s}) \le k-s \}.
\]
\end{theorem}
\begin{proof}
To prove the ``$\supseteq$'' part, let be $\Xi$ an element from the set on the right side of the equality. Then $\Xi = \Xi_s + \Xi_{k-s}$ with $\Xi_s \in T_X \mM_s$, and $\rank(\Xi_{k-s}) \le k-s$. There exist a sequence $(Y_n) \subseteq \mM_s$ and a sequence $(a_n) \subseteq \R^+$ such that $Y_n \to X$, and $a_n(Y_n - X) = \Xi_s$. One can assume $a_n \to \infty$. Then $X_n^{} = Y_n^{} + a_n^{-1} \Xi_{k-s}$ is a sequence in $\mM_{\le k}$ which converges to $X$, and $a_n(X_n - X)$ converges to $\Xi$, which proves $\Xi \in T_X \mM_{\le k}$.

To prove the reverse inclusion ``$\subseteq$'', assume $\Xi = \lim_{n \to \infty} a_n (X - X_n)$, $X_n \to X$ in $\mM_{\le k}$, and $(a_n) \subseteq \R^+$. In the orthogonal decomposition
\[
a_n(X_n - X) = \Pi_{T_X \mM_s} a_n(X_n - X) + (\Pi_{\mU^\bot} \otimes \Pi_{\mV^\bot})a_n X_n,
\]
both terms have to converge separately. Denote their limits by $\Xi_s$ and $\Xi_{k-s}$, respectively. Then obviously $\Xi = \Xi_s + \Xi_{k-s}$ with $\Xi_s \in T_X \mM_s$ and $\Xi_{k-s} \in \mU^\bot \otimes \mV^\bot$. Since $(\Pi_\mU \otimes \Pi_\mV)X_n \to (\Pi_\mU \otimes \Pi_\mV)X = X$, and since the set of rank-$s$ matrices is relatively open in $\mU \otimes \mV$, $\rank((\Pi_\mU \otimes \Pi_\mV)X_n) = s$ for large enough $n$. Consequently, since $\rank(X_n) \le k$ for all $n$, it must hold that $\rank((\Pi_{\mU^\bot} \otimes \Pi_{\mV^\bot})a_n X_n) = \rank((\Pi_{\mU^\bot} \otimes \Pi_{\mV^\bot})X_n) \le k-s$ for large enough $n$. It follows from the semicontinuity of matrix rank that $\rank(\Xi_{k-s}) \le k-s$.
\qquad\end{proof}

\emph{Remark.} In the recent paper~\cite{CasonAbsilVanDooren2013} the tangent cones of $\mM_{\le k}$ have been previously derived, but in contrast to~\eqref{eq: definition of tangent cone} are defined via analytic curves as
\begin{equation}\label{eq: alternative definition of tnagent cone}
T_X \mM_{\le k} = \{ \dot{\gamma}(0) \vcentcolon \text{$\gamma$ is an analytic curve with  $\gamma(0) = X$ and $\gamma(t) \in \mM_{\le k}$ for $t \ge 0$} \}.
\end{equation}
As shown in~\cite[Proposition~2]{OSheaWilson2004}, both definitions are equivalent. Up to an additional normalization constraint, the authors of~\cite{CasonAbsilVanDooren2013} essentially prove Theorem~\ref{prop: form of tangent cone} using definition~\eqref{eq: alternative definition of tnagent cone}, which together with our proof provides a direct verification that both definitions are equivalent. As mentioned in~\cite{CasonAbsilVanDooren2013}, when using definition~\eqref{eq: alternative definition of tnagent cone}, the ``$\subseteq$'' part in Theorem~\ref{prop: form of tangent cone} follows from known results on the existence of analytic ``singular value decomposition paths''~\cite{BunseGerstneretal1991}. We can easily modify our argument above to prove the ``$\supseteq$'' part for~\eqref{eq: alternative definition of tnagent cone} by choosing an analytic curve $\gamma_s$ in $\mM_s$ (possible by Lemma~\ref{lem: M_s real-analytic submanifold}) such that $\Xi_s = \dot{\gamma}_s(0)$, and put $\gamma(t) = \
\gamma_s(t) + t \Xi_{k-s}$, which is an analytic curve in $\mM_{
\le k}$ with $\dot{\gamma}(0) = \Xi = \Xi_s + \Xi_{k-s}$. The proof of ``$\supseteq$'' given in~\cite{CasonAbsilVanDooren2013} seems more involved than is probably necessary, since the well-known structure of $T_X \mM_s$ is not exploited.

\emph{Remark.} In~\cite{Luke2013}, formulas for normal cones of $\mM_{\le k}$ have been derived. They do not imply the formula for the tangent cone in singular points $X$ with $s<k$. The reverse, however, is true. In view of~\eqref{eq: optimality condition}, we can rephrase Corollary~\ref{cor: deeper insight} below by stating that the regular normal cone at such $X$ contains only zero. This then implies that the general normal cone~\cite[Definition~6.3]{RockafellarWets1998} at $X$ is the union of all limits of subspaces $(T_{X_n} \mM_{\le k})^\bot = \mU_n^\bot \otimes \mV_n^\bot$ with $X_n \to X$ and $\rank(X_n) = k$. Consequently, the singular points of $\mM_{\le k}$ are also not regular in the sense of Clarke~\cite[Definition~6.4]{RockafellarWets1998}.

Now that we know the structure of the tangent cone in rank-deficient points, we can calculate the projection of the antigradient on it. This turns out to be easy. Moreover, the tangent cone in such points is so ``large'' that the projection on it carries over astonishingly much information. In fact, it generates all of $\R^{m \times n}$.

\begin{corollary}\label{cor: projection of negative gradient on tangent cone}
Let $X \in \mM_{\le k}$, $\rank(X) = s$. Any $G \in T_X \mM_{\le k}$ satisfying $\| - \nabla f(X) - G\|_\frob = \dist_\frob(-\nabla f(X),T_X \mM_{\le k})$ has the form
\begin{equation}\label{eq: form of G}
G = \Pi_{T_X \mM_s} (- \nabla f(X)) + \Xi_{k-s},
\end{equation}
where $\Xi_{k-s}$ is a best rank-$(k-s)$ approximation of $(\Pi_{\mU^\bot} \otimes \Pi_{\mV^\bot})( - \nabla f(X)) = -\nabla f(X) - \Pi_{T_X \mM_s}(- \nabla f(X))$ in the Frobenius norm. (Obviously, $\Xi_{k-s} \in \mU^\bot \otimes \mV^\bot$ then.) Moreover,
\begin{equation}\label{eq: lower estimate for projected gradient}
g^-(X) = \| G \|_\frob \ge \sqrt{\frac{k-s}{\min(m-s,n-s)}}\| \nabla f(X) \|_\frob.
\end{equation}
\end{corollary}

\begin{proof}
The form of $G$ is clear from Theorem~\ref{prop: form of tangent cone} by orthogonality considerations. We prove the norm estimate. The square of the Frobenius norm of a matrix is the sum of its squared singular values. A best rank-$(k-s)$ approximation of a matrix in the Frobenius norm is obtained by truncating its singular value decomposition up to the largest $k-s$ singular values. As $\dim(\mU^\bot) = m-s$ and $\dim(\mV^\bot) = n-s$, the matrix $(\Pi_{\mU^\bot} \otimes \Pi_{\mV^\bot})(- \nabla f(X))$ has at most $\min(m-s,n-s)$ nonzero singular values.
We conclude that
\[
\| \Xi_{k-s} \|_\frob^2 \ge  \frac{k-s}{\min(m-s,n-s)} \| (\Pi_{\mU^\bot} \otimes \Pi_{\mV^\bot}) \nabla f(X) \|_\frob^2.
\]
Since $\Xi_{k-s} \in \mU^\bot \otimes \mV^\bot$,~\eqref{eq: form of G} and~\eqref{eq: orthogonal decomposition} now show that
\begin{align*}
\| G \|_\frob^2 &= \| \Pi_{T_X \mM_s} (- \nabla f(X)) \|_\frob^2 + \| \Xi_{k-s}\|_\frob^2 \\
&\ge \frac{k-s}{\min(m-s,n-s)} ( \| \Pi_{T_X \mM_s} (- \nabla f(X)) \|_\frob^2 + \| (\Pi_{\mU^\bot} \otimes \Pi_{\mV^\bot}) (-\nabla f(X)) \|_\frob^2)\\
&= \frac{k-s}{\min(m-s,n-s)} \| -\nabla f(X) \|_\frob^2, 
\end{align*}
as asserted.
\qquad\end{proof}

The estimate~\eqref{eq: lower estimate for projected gradient} allows us to make a remarkable a priori statement about critical points of differentiable functions on $\mM_{\le k}$.

\begin{corollary}\label{cor: deeper insight}
Let $k \le \min(m,n)$, and let $X^* \in \mM_{\le k}$ be a critical point of~\eqref{eq: low-rank min problem} in the sense $g^-(X^*) = 0$. Then either $\rank(X^*) = k$ or $\nabla f(X^*) = 0$.
\end{corollary}

\changed{Conceptually similar statements have been made in~\cite[Proposition 4]{BurerMonteiro2003} and~\cite[Theorem 7]{Journeeetal2010} for optimization tasks on the set of positive semidefinite matrices.} As an illustration consider the following.

\begin{corollary}
Let $k \le \min(m,n)$. Assume that $f \vcentcolon \R^{m \times n} \to \R$ is strictly convex and coercive and its unique minimizer on $\R^{m \times n}$ has rank larger than or equal to $k$. Then any relative local minimizer of $f$ on $\mM_{\le k}$ has rank $k$.
\end{corollary}

In light of these results, it is not surprising that we will have to make the assumption $\rank(X^*) = k$ in our convergence results below in order to conclude $g^-(X^*) = 0$. It is not an artifact of the used techniques. Instead, Corollary~\ref{cor: deeper insight} tells us that it will be normally impossible to find a rank-deficient critical point by a projected gradient method that most of the time moves on $\mM_k$, since on $\mM_{k}$ the projection of the antigradient contains much less information. 

We finish with a practical remark. When the matrices are large, one will only be able to work with sparse or low-rank representations of all involved quantities. In particular, $\nabla f(X)$ needs to allow for a sparse or a low-rank representation. If $s \ll \min(m,n)$, the calculation of $\Pi_{\mU^\bot} \otimes \Pi_{\mV^\bot}(-\nabla f(X)) = -\nabla f(X) - \Pi_{T_X \mM_s}(- \nabla f(X))$ is then feasible using the second representation of $\Pi_{T_X \mM_s}$ in~\eqref{eq: projector on tangent space}. With some effort one can even exploit the low-rank structure of $\Pi_{T_X \mM_s}(- \nabla f(X))$ to calculate an approximate singular value decomposition of the difference without explicitly assembling it. The huge projector $\Pi_{\mU^\bot} \otimes \Pi_{\mV^\bot}$ should never be formed. The final rank of tangent vectors itself is not larger than $2s + (k-s) = k+s$, which can be seen from the decomposition. We summarize the procedure as Algorithm~\ref{Alg: Calculate gradient}.

\begin{algorithm}
\KwIn{Antigradient $F = - \nabla f(X)$ at $X \in \mM_{\le k}$.}
\KwOut{Projection $G \in T_X \mM_{\le k}$ with $\| F - G\|_\frob = \dist_\frob(F,T_X \mM_{\le k})$.}
Find orthonormal bases $U$ and $V$ for $\ran(X)$ and $\ran(X^\trans)$, respectively\;
Calculate the projection on $T_X \mM_s$:
\[
	\Xi_s = UU^\trans F + F VV^\trans - UU^\trans F VV^\trans;
\]
\nl Calculate best rank-$(k-s)$ approximation of the difference:
\[
\Xi_{k-s} \in \argmin_{\rank(Y) \le k-s} \| F - \Xi_s - Y\|_\frob;
\]
\nl Output:
\[
G = \Xi_s + \Xi_{k-s}, \quad g^-(X) = \| G \|_\frob = \sqrt{\|\Xi_s\|_\frob^2 + \|\Xi_{k-s}\|_\frob^2}.
\]
\caption{Calculate the projection of $- \nabla f(X)$ on $T_X \mM_{\le k}$}\label{Alg: Calculate gradient}
\end{algorithm}

\subsection{Retraction by best low-rank approximation}\label{sec: stable retraction for matrices}

As a retraction we choose the best approximation by a matrix of rank at most $k$ in the Frobenius norm, i.e.,
\begin{equation}\label{eq: best approximation retraction}
R(X, \Xi_X) \in \argmin_{Y \in \mM_{\le k}} \|Y - (X + \Xi_X)\|_\frob. 
\end{equation}
It can be explicitly calculated using singular value decomposition. In unlikely events,~\eqref{eq: best approximation retraction} is set-valued, but we can assume that a specific choice is made by fixing deterministic singular value decomposition and truncation algorithms. The particular choice does not matter. We emphasize once more that Definition~\ref{def: retraction} is indeed fulfilled: let $\Xi \in T_X \mM_{\le k}$; then by~\cite[Proposition~2]{OSheaWilson2004} there exists an analytic arc $\gamma \vcentcolon [0,\epsilon) \to \mM_{\le k}$ such that $\dot{\gamma}(0) = \Xi$. Hence,
\begin{equation}\label{eq: valid retraction on variety}
\lim_{\alpha \to 0^+} \frac{\| R(X, \alpha \Xi_X) - (X + \Xi_X) \|_\frob}{\alpha} \le \lim_{\alpha \to 0^+} \frac{\| \gamma(\alpha) - (X + \dot{\gamma}(0)) \|_\frob}{\alpha} = 0.
\end{equation}

We have the following nice estimate, which provides $M = 1 + 2^{-1/2}$ in~\eqref{eq: upper bound for retraction}.

\begin{proposition}\label{prop: stable retraction}
The above retraction satisfies
\[
\| R(X, \Xi_X) - (X + \Xi_X) \|_\frob \le \frac{1}{\sqrt{2}} \| \Xi_X \|_\frob \quad \text{for all $X \in \mM_{\le k}$ and $\Xi_X \in T_X \mM_{\le k}$.}
\]
\end{proposition}
\begin{proof}
The matrices
\(
X + (\Pi_\mU \otimes I)\Xi_X = (\Pi_\mU \otimes I)(X + \Xi_X)
\)
and
\(
X + (I \otimes \Pi_\mV)\Xi_X = (I \otimes \Pi_\mV)(X + \Xi_X)
\)
both have rank at most $s$. Thus, by Theorem~\ref{prop: form of tangent cone},
\[
X + (\Pi_\mU \otimes I)\Xi_X + (\Pi_{\mU^\bot} \otimes \Pi_{\mV^\bot})\Xi_X
\]
and
\[
X + (I \otimes \Pi_\mV)\Xi_X + (\Pi_{\mU^\bot} \otimes \Pi_{\mV^\bot})\Xi_X
\]
both have rank not larger than $k$. Considering them as possible candidates for a best approximation $R(X, \Xi_X)$ of $X + \Xi_X$ by a matrix of rank at most $k$, we obtain the desired bound
\[
\| R(X + \Xi_X) - (X + \Xi_X) \|_\frob^2 \le \min(\|(\Pi_{\mU^\bot} \otimes \Pi_\mV)\Xi_X \|_\frob^2, \|(\Pi_\mU \otimes \Pi_{\mV^\bot})\Xi_X\|_\frob^2)  \le \frac{1}{2} \| \Xi_X \|_\frob^2,
\]
where we have made use of the orthogonal decompositions~\eqref{eq: tangent space Ms} and~\eqref{eq: orthogonal decomposition}.
\qquad\end{proof}

We conclude that $\| R(X, \Xi_X) - X \|_\frob^2 \to 0$ automatically implies $\Xi_X \to 0$.

\subsection{\L{}ojasiewicz inequality and convergence result}\label{sec: lojasiewicz for matrices}

To apply the convergence results of section~\ref{eq: relation to line-search}, we will show that the \L{}ojasiewicz gradient inequality~\eqref{L} holds for real-analytic functions in every point of $\mM_{\le k}$. The aim is to apply Proposition~\ref{prop: Lojasiewicz for real-analytic manifolds}. \changed{
As a first step, the following lemma implies that the smooth submanifolds $\mM_s$ are indeed real-analytic submanifolds in the sense of~\cite[Definition 2.7.1]{KrantzParks2002}.

\begin{lemma}\label{lem: M_s real-analytic submanifold}
Let $0 < s \le \min(m,n)$. The (componentwise) real-analytic map
\[
(U,V) \mapsto UV^\trans
\]
is a submersion from the open subset $\{(U,V) \in \R^{m \times s} \times \R^{n \times s} \vcentcolon \rank(U) = \rank(V) = s \}$ of $\R^{m \times s} \times \R^{n \times s}$ onto $\mM_s$.
\end{lemma}
\begin{proof}
The openness of the domain of definition and the surjectivity are clear. The derivative at $(U,V)$ is the linear map $(\delta U, \delta V) \mapsto \delta U V^\trans + U \delta V^\trans$. As both $U$ and $V$ have rank $s$, one verifies that it has no nontrivial kernel in the $(m+n-s)s$-dimensional subspace of all $(\delta U, \delta V)$ satisfying $U^\trans \delta U = 0$. Hence the derivative has rank at least $(m+n-s)s = \dim(\mM_s)$ (see Theorem~\ref{th: smoothness of full-rank manifolds}), which already proves the claim. \qquad
\end{proof}
}

\begin{theorem}\label{th: Lojasiewicz for matrices}
Let $\mD \subseteq \R^{m \times n}$ be open, $\mM_{\le k} \subset \mD$, and $f \colon \mD \to \R$ be real-analytic. Then the \L{}ojasiewicz gradient inequality~\eqref{L} holds at any point $X \in \mM_{\le k}$.
\end{theorem}
\changed{
\begin{proof}
Let $X \in \mM_{\le k}$, $\rank(X) = s$. We assume $s > 0$; otherwise the proof requires some obvious notational modifications. There exist matrices $U_0 \in \R^{m \times s}$ and $V_0 \in \R^{n \times s}$, both of rank $s$, such that $X = U_0^{} V_0^\trans$. We consider the map
\begin{gather*}
\tau \vcentcolon \mN = \R^{m \times s} \times \R^{n \times s} \times \R^{m \times (k-s)} \times \R^{n \times (k-s)} \to \R^{m \times n}, \\
(U,V,U_{k-s},V_{k-s}) \mapsto U V^\trans + U_{k-s}^{} V_{k-s}^\trans.
\end{gather*}
This map is obviously real-analytic, its image is $\mM_{\le k}$, and $\tau(U_0,V_0,0,0) = X$. We need to prove property (ii) in Proposition~\ref{prop: Lojasiewicz for real-analytic manifolds}. Let $\widetilde{\mN} \subseteq \mN$ be an open neighborhood of $(U_0,V_0,0,0)$. We may assume that this neighborhood is so small such that for all $(U,V,U_{k-s},V_{k-s}) \in \widetilde{\mN}$ it holds that $\rank(U) = \rank(V) = s$. Lemma~\ref{lem: M_s real-analytic submanifold} then implies that the map $(U,V,U_{k-s},V_{k-s}) \mapsto UV^\trans$ is a submersion from $\widetilde{\mN}$ on the smooth manifold $\mM_s$, and as such an open map~\cite[\S(16.7.5)]{DieudonneIII}. 
Consequently, for small enough 
$\delta > 0$ we can claim that $\tau(\widetilde{\mN})$ contains \emph{all} matrices of the form
\(
Y_s + Y_{k-s}
\)
with $\rank(Y_s) = s$, $\rank(Y_{k-s}) \le k-s$, and $\| Y_s - X \|_\frob < 2\delta$. By semicontinuity of rank, we can choose $\delta$ so small that $\rank(Y) \ge s$ for all 
$Y \in B_\delta(X)$, where $B_\delta(X)$ denotes the open ball in $\R^{m \times n}$ of radius $\delta$ (in Frobenius norm) with center $X$. Let $Y_s$ denote the best rank-$s$ approximation (in Frobenius norm) of $Y \in B_\delta(X) \cap \mM_{\le k}$. As $Y_s$ is obtained by truncating a singular value decomposition of $Y$, we have $Y = Y_s + Y_{k-s}$ with $\rank(Y_{k-s}) \le k-s$, and
\[
 \| Y_s - X \|_\frob \le \| Y_s - Y \|_\frob + \| Y - X \|_\frob \le \| X - Y \|_\frob  + \| Y - X \|_\frob < 2\delta,
\]
the second inequality holding since $\rank(X) = s$. By the previous considerations, this implies $Y \in \tau(\widetilde{\mN})$. We thus have shown that $\tau(\widetilde{\mN})$ contains the relatively open set $B_\delta(X) \cap \mM_{\le k}$.
\qquad\end{proof}
}

We now have collected all requirements to apply Theorem~\ref{th: abstract theorem for line-search} or Corollary~\ref{cor: smooth line-search result}. For concreteness, we consider a particular algorithm where the search direction equals the projected antigradient and the retraction is obtained by best rank-$k$ approximation. It is denoted as Algorithm~\ref{Alg: SD for matrix manifold}.

\begin{algorithm}
\KwIn{Starting guess $X_0 \in \mM_{\le k}$, $\beta,c \in (0,1)$.}
\For{n=0,1,2,\dots}{
Calculate a projection $G_n$ of $- \nabla f(X_n)$ on $T_{X_n} \mM_{\le k}$ using Algorithm~\ref{Alg: Calculate gradient}\;
Choose $\bbeta_n \ge 1$, and find Armijo point $\alpha_n$ for $X_n, G_n, \bbeta_n, \beta, c$\;
Set $X_{n+1}$ to be a best approximation (with respect to Frobenius norm) of $X_n + \alpha_n G_n$ of rank at most $k$.
}
\caption{Projected steepest descent with line-search on $\mM_{\le k}$}\label{Alg: SD for matrix manifold}
\end{algorithm}

\begin{theorem}\label{th: result for low-rank steepest descent}
Let $f$ be real-analytic and bounded below. If the sequence $(X_n)$ generated by Algorithm~\ref{Alg: SD for matrix manifold} possesses a cluster point $X^*$, then it is its limit. If further $\rank(X^*) = k$, then $g^-(X^*) = 0$, and the convergence rate estimates of Theorem~\ref{th: main theorem} apply.
\end{theorem}
\begin{proof}
The convergence of the sequence follows from Theorem~\ref{th: Lojasiewicz for matrices}, Proposition~\ref{prop: existence of decrease direction}, and Corollary~\ref{cor: mere convergence of line-search}. Due to Theorem~\ref{th: smoothness of full-rank manifolds}, the rest is an instance of Corollary~\ref{cor: smooth line-search result}. 
\qquad\end{proof}

\subsection{A method without retraction}\label{sec: a method without retraction}

It is possible to have a gradient-related search direction $\Xi_n$ such that $X_n + \alpha \Xi_n \in \mM_{\le k}$ for all $\alpha$. The idea is the same as in the proof of Proposition~\ref{prop: stable retraction}. By~\eqref{eq: form of G} and~\eqref{eq: tangent space Ms}, a projection $G$ of $- \nabla f_n$ consists of four, mutually orthogonal parts: 
\begin{equation*}\label{eq: four parts}
G_n = (\Pi_\mU \otimes \Pi_\mV)(-\nabla f_n) + (\Pi_{\mU^\bot} \otimes \Pi_\mV)(-\nabla f_n) + (\Pi_\mU \otimes \Pi_{\mV^\bot})(-\nabla f_n) + \Xi_{k-s,n},
\end{equation*}
with $\rank(\Xi_{k-s,n}) \le k-s$. Consider the two possible partial projections
\begin{equation}\label{eq: first candidate}
G_n^{(1)} = (\Pi_\mU \otimes \Pi_\mV)(-\nabla f_n) + (\Pi_\mU \otimes \Pi_{\mV^\bot})(-\nabla f_n) + \Xi_{k-s,n} = (\Pi_\mU \otimes I)(-\nabla f_n) + \Xi_{k-s,n}
\end{equation}
and
\begin{equation}\label{eq: second candidate}
G_n^{(2)} = (\Pi_\mU \otimes \Pi_\mV)(-\nabla f_n) + (\Pi_{\mU^\bot} \otimes \Pi_\mV)(-\nabla f_n) + \Xi_{k-s,n} = (I \otimes \Pi_\mV)(-\nabla f_n) + \Xi_{k-s,n}.
\end{equation}
Both are elements of the tangent cone at $X_n$ and satisfy $\rank(X_n + \alpha G_n^{(i)}) \le k$ for all $\alpha$, $i=1,2$. Assume that $\| G_n^{(1)} \|_\frob \ge \| G_n^{(2)} \|_\frob$. Then, by orthogonality arguments, $\| G_n^{(1)} \|_\frob^2 \ge \frac{1}{2}\| G_n \|_\frob^2$, and
\[
\langle \nabla f_n, G_n^{(1)}  \rangle_\frob = \| G_n^{(1)}  \|_\frob^2 \ge \frac{1}{\sqrt{2}} \| G_n \|_\frob \| G_n^{(1)}  \|_\frob = \frac{1}{\sqrt{2}} g_n^- \| G_n^{(1)}  \|_\frob.
\]
Thus $G_n^{(1)}$ satisfies the $\omega$-angle condition with $\omega = \frac{1}{\sqrt{2}}$. If $\| G_n^{(1)} \|_\frob \le \| G_n^{(2)} \|_\frob$, then $G_n^{(2)}$ satisfies this angle condition.

This leads us to Algorithm~\ref{Alg: without retraction}, which contains no retraction steps. Still, it shares the nice abstract convergence features with the projected steepest descent, even with a slightly extended statement in singular points (convergence rate).

\begin{algorithm}
\KwIn{Starting guess $X_0 \in \mM_{\le k}$, $\beta,c \in (0,1)$.}
\For{n=0,1,2,\dots}{
\uIf{$\| (\Pi_\mU \otimes I) (- \nabla f(X_n))\|_\frob \ge \| (I \otimes \Pi_\mV)( - \nabla f(X_n))\|_\frob$}{
Use $\Xi_n = G_n^{(1)}$ from~\eqref{eq: first candidate}\;
}
\Else{
Use $\Xi_n = G_n^{(2)}$ from~\eqref{eq: second candidate}\;
}
Choose $\bbeta_n \ge \sqrt{2}$, and find Armijo point $\alpha_n$ for $X_n, \Xi_n, \bbeta_n, \beta, c$\;
Form the next iterate
\[
X_{n+1} = X_n + \alpha_n \Xi_n.
\]
}
\caption{Descent method on $\mM_{\le k}$ without retraction}\label{Alg: without retraction}
\end{algorithm}

\begin{theorem}\label{th: result for method without retraction}
Let $f$ be real-analytic and bounded below. If the sequence $(X_n)$ generated by Algorithm~\ref{Alg: without retraction} possesses a cluster point $X^*$, then it is its limit, and the convergence rate estimates of Theorem~\ref{th: main theorem} apply. If further $\rank(X^*) = k$, then $g^-(X^*) = 0$.
\end{theorem}
\begin{proof}
Since $\rank(X_n + \alpha_n \Xi_n) \le k$, we can formally write $X_{n+1} = R(X_n, \alpha_n \Xi_n)$ in the algorithm in order to get into the abstract framework (here $R$ is again retraction by best low-rank approximation). Then the mere convergence of the sequence follows again from Theorem~\ref{th: Lojasiewicz for matrices} and Corollary~\ref{cor: mere convergence of line-search}. The feature is now that~\eqref{eq: abstract property} is trivially satisfied since $R$ acts as identity; therefore the validity of convergence rate estimates follows from Theorem~\ref{th: abstract theorem for line-search} even if the limit point is singular (the Lipschitz condition~\eqref{eq: Lipschitz condition} follows from analyticity). We also have $g_n^-(X_n) \to 0$ from which we can conclude $g^-(X^*) = 0$ if $g^-$ is continuous in $X^*$. But this is the case if $X^* \in \mM_k$.
\qquad\end{proof}

Since it does not leave the feasible set, Algorithm~\ref{Alg: without retraction} is very elegant and saves some cost in every step of the backtracking to find the Armijo point. In applications, however, the retraction from rank (at most) $2k$ to rank $k$, as required in Algorithm~\ref{Alg: SD for matrix manifold}, is typically much less expensive than, for instance, a function value evaluation or the projection of the gradient. We hence expect that the saved retractions will seldom compensate for the less gradient-related search directions.

We checked this with a toy example of matrix completion in a setup similar to~\cite{Vandereycken2013}, using \emph{straightforward, comparably nonoptimized} MATLAB R2012b implementations of both algorithms (choosing $\beta = \frac{1}{2}$ and $c = 10^{-4}$) on a Linux workstation with six 3.2 GHz CPU cores and 6 GB of memory. The problem that was solved is
\begin{equation}\label{eq: matrix completion problem}
\min_{X \in \mM_{\le k}} \frac{1}{2} \| P_\Omega ( A - X ) \|_\frob^2,
\end{equation}
where $P_\Omega$ is the projector on a subset $\Omega$ of indices. The $n \times n$ matrix $A = UV^\trans$ of rank $r$ was generated by randomly generating the two $n \times r$ factor matrices $U$ and $V$ from a normal distribution. The size of $\Omega$ was chosen as $| \Omega | = \max(\mathtt{OS} \cdot (2kn - k^2), n \log n)$, which corresponds to an oversampling rate of at least $\mathtt{OS}$ when assuming $A$ to have rank $k$ (cf.~\cite{Vandereycken2013}), and $\Omega$ itself was drawn uniformly at random. As a starting guess we chose in all experiments a best rank-$k$ approximation of the antigradient $-\nabla f(0) = -P_\Omega(A)$. In both Algorithms~\ref{Alg: SD for matrix manifold} and~\ref{Alg: without retraction}, this choice of starting guess is formally equivalent to starting with zero and performing an exact line-search in the very first step.

In the first test the rank of $A$ was indeed set to be $r = k$, so that the global solution of~\eqref{eq: matrix completion problem} lies on the smooth part $\mM_k$ of $\mM_{\le k}$. For $n = 2000$, $k = 20$, and $\mathtt{OS} = 3$ ($94.03 \%$ missing entries), the relative errors,
\begin{equation}\label{eq: full error}
\frac{\| A - X_n \|_\frob}{\| A \|_\frob},
\end{equation}
as well as the relative errors on the visible index set,
\begin{equation}\label{eq: sample error}
\frac{\| P_\Omega ( A - X_n ) \|_\frob}{\| P_\Omega A \|_\frob} = \frac{\sqrt{2 f(X_n)}}{\| P_\Omega A \|_\frob},
\end{equation}
are plotted in Figure~\ref{fig: Alg3vsAlg4}. As one can see, Algorithm~\ref{Alg: without retraction} is inferior to Algorithm~\ref{Alg: SD for matrix manifold} with respect to both number of iterations and computation time (the latter is plotted just to give an impression). One might think that the relative performance of Algorithm~\ref{Alg: without retraction} improves for larger $k$. The plots for $k = 80$ do not support this hope (in this case only $76.48 \%$ entries are missing, which perhaps explains the faster error decay). \changed{On the other hand, we did not observe that the superiority of Algorithm 3 would become considerably more pronounced for larger matrices; plots looked very similar.}

\begin{figure}
\includegraphics[width=\textwidth]{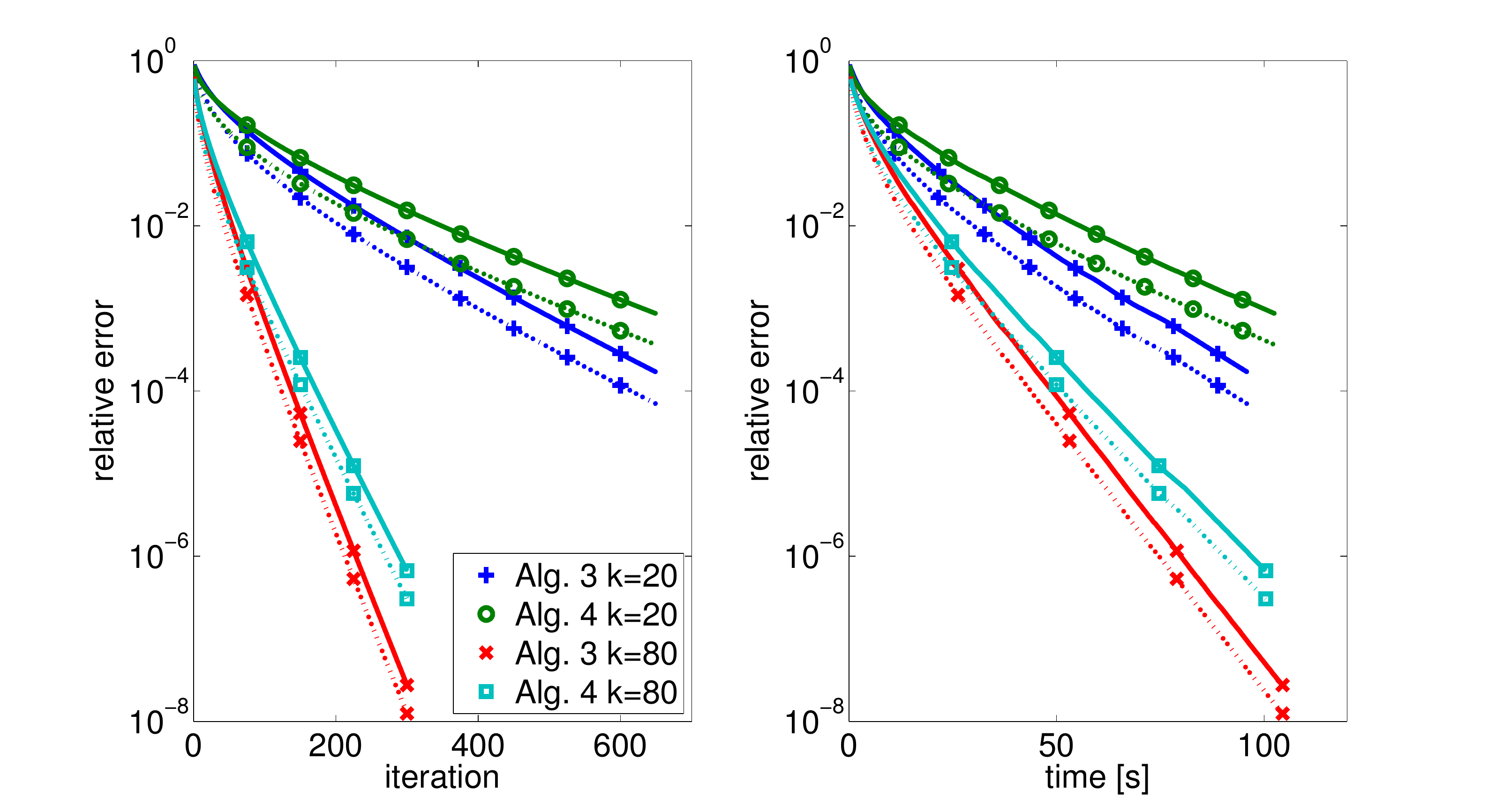}
\caption{Application of Algorithms~\ref{Alg: SD for matrix manifold} and~\ref{Alg: without retraction} to~\eqref{eq: matrix completion problem}  with $A \in \R^{2000 \times 2000}$, $\rank(A) = k$, for $k=20$ ($94.03 \%$ missing entries), and $k=80$ ($76.48 \%$ missing entries). Solid lines: relative errors~\eqref{eq: full error} (full index set); dashed lines: relative errors~\eqref{eq: sample error} (sample index set).}
\label{fig: Alg3vsAlg4}
\end{figure}

\changed{
We repeated the same experiments with matrices $A$ having rank $r = k/2$. Interestingly, Algorithm~\ref{Alg: without retraction} now performed better than Algorithm~\ref{Alg: SD for matrix manifold}, but both methods were unable to find a good approximation of the rank-deficient global solution $A$; see Figure~\ref{fig: Alg3vsAlg4rankdeficient}. In fact, we practically never encountered iterates whose $k$th singular value was less than $10^{-4}$. This confirms our expectation that our approach via line-search methods on $\mM_{\le k}$ does not contribute to the problem of rank estimation. Yet it served as an elegant theoretical vehicle to prove the convergence of Riemannian line-search methods on $\mM_k$ in ``almost every instance.'' As indicated in the introduction, a possible synthesis are rank-increasing strategies which subsequently optimize on varieties $\mM_{\le s}$ for a growing sequence of $s$~\cite{Mishraetal2013,Tanetal2014,UschmajewVandereycken2014}. 
}

Of course, Algorithms~\ref{Alg: SD for matrix manifold} and~\ref{Alg: without retraction} served here only as examples and are naturally inferior to more sophisticated line-search methods, such as the nonlinear CG methods used in~\cite{Vandereycken2013}, which use gradient information from previous iterates.

\begin{figure}
\includegraphics[width=\textwidth]{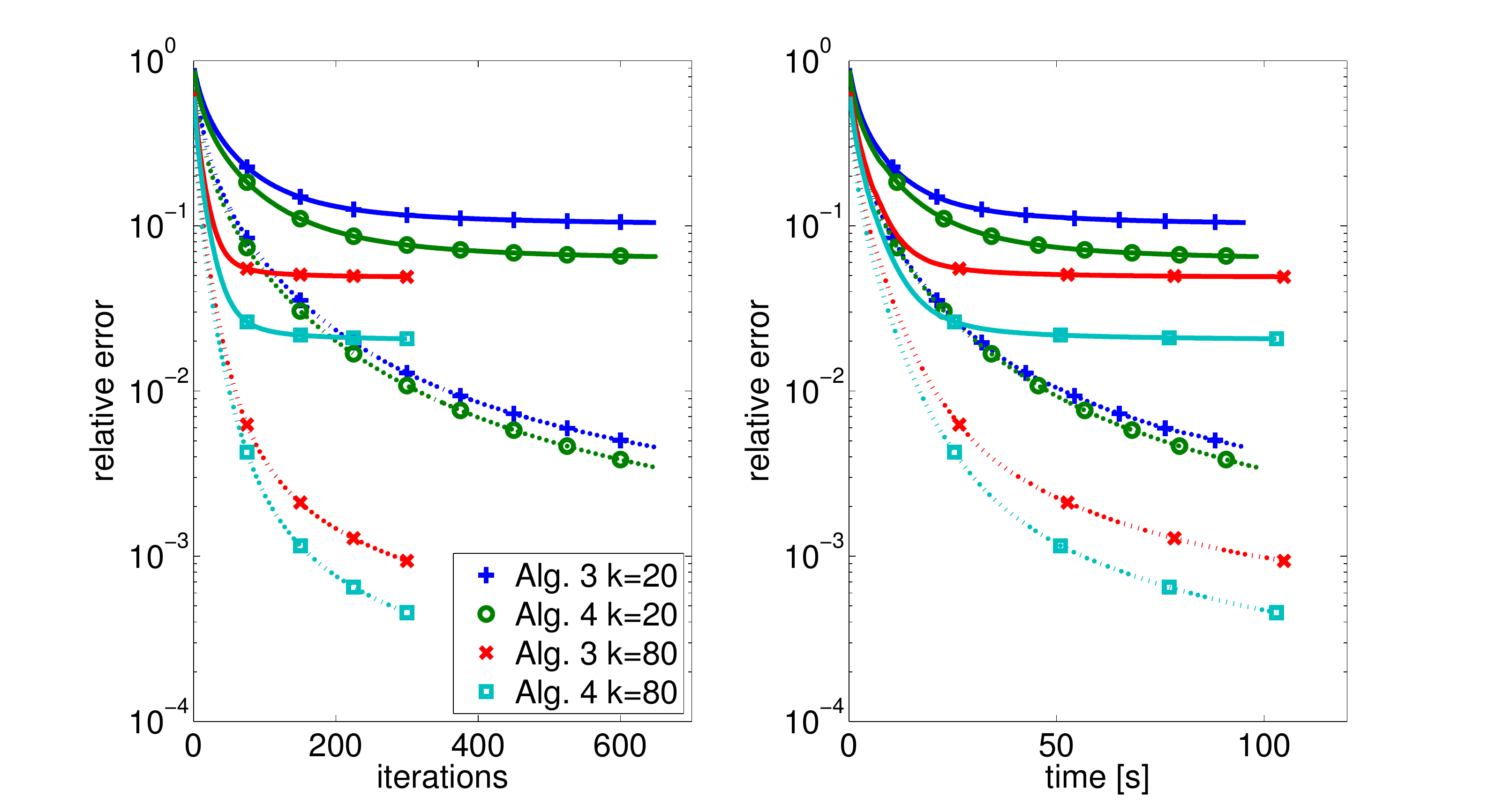}
\caption{Application of Algorithms~\ref{Alg: SD for matrix manifold} and~\ref{Alg: without retraction} to~\eqref{eq: matrix completion problem}  with $A \in \R^{2000 \times 2000}$, $\rank(A) = k/2$, for $k=20$ ($94.03 \%$ missing entries), and $k=80$ ($76.48 \%$ missing entries). Solid lines: relative errors~\eqref{eq: full error} (full index set); dashed lines: relative errors~\eqref{eq: sample error} (sample index set).}
\label{fig: Alg3vsAlg4rankdeficient}
\end{figure}

\section{Conclusion}

We extended available results on convergence of descent iterations on manifolds via \L{}ojasiewicz gradient inequality to gradient-related line-search methods on the real-algebraic variety $\mM_{\le k}$ of real $m\times n$ matrices of rank at most $k$, by explicitly taking the tangent cones at singular points into consideration. This made it possible to overcome some theoretical difficulties arising from the nonclosedness and unbounded curvature that one faces in the convergence analysis of Riemannian optimization methods on the smooth manifold $\mM_k$ of rank-$k$ matrices. So far, the results are applicable for real-analytic cost functions.

There is growing interest in treating low-rank tensor problems by Riemannian optimization, e.g., tensor completion~\cite{KressnerSteinlechnerVandereycken2013} or dynamical tensor approximation~\cite{KochLubich2010,Lubichetal2013,UschmajewVandereycken2013}. It would be important and interesting to extend the results to tensor varieties of bounded subspace ranks, e.g., bounded Tucker ranks, hierarchical Tucker ranks, or tensor train ranks~\cite{KoldaBader2009,HackbuschKuehn2009,Oseledets2011}. As these varieties take the form of intersections of low-rank matrix varieties~\cite{UschmajewVandereycken2013}, the results in this paper can likely be generalized in this direction.

\section*{Acknowledgments}

We thank Pierre-Antoine Absil who suggested an improvement of Corollary~\ref{cor: projection of negative gradient on tangent cone}, and Bart Vandereycken for useful hints to the literature.

\Appendix

\section{Proof of Theorem~\ref{th: main theorem}}

We can assume that $ g^-_n >  0$ for all $n$ since otherwise the sequence becomes stationary and there is nothing to prove. There will also be no loss of generality to assume that~\eqref{A1} and~\eqref{A2} hold for all $n$ and that $f(x^*)  = 0$. Then $0 \le f(x^*) \le f_n$ for all $n$, and the \L{}ojasiewicz gradient inequality at $x^* $ reads as
\begin{equation}\label{eq: Loj ineq in proof}
  f(x)^{1 - \theta} \leq \Lambda g^-(x)  
\end{equation}
whenever $\|x - x^* \| < \delta=\delta(x^*)$. Let $\epsilon \in (0,\delta]$, and assume $\|x_n - x^* \|  < \delta$. Then, by~\eqref{eq: Loj ineq in proof} and~\eqref{A1},
\[
\|x_{n} - x_{n+1} \| \le    \frac{\Lambda}{\sigma }     f_n^{\theta -1  } (  f_n - f_{n+1}).
\]
Using the fact that for $\varphi \in [f_{n+1}, f_n]$  there holds  $f_n^{\theta -1 } \leq \varphi^{\theta -1 } \leq f_{n+1}^{\theta -1 }$, we can estimate
\[
f_n^{\theta -1  } ( f_n - f_{n+1} ) \le \int_{f_{n+1}}^{f_n}  \varphi^{\theta -1 }\, \mathrm{d}\varphi = \frac{1}{\theta } (   f_n^{\theta}  - f_{n+1}^{\theta }  )
\]
and thus obtain 
\[
\|x_n - x_{n+1} \| \leq    \frac{\Lambda }{\sigma \theta }(   f_n^{\theta}  - f_{n+1}^{\theta }).
\]
More generally, let $\| x_k - x^* \| < \epsilon \le \delta$ for $n \le k < m$; we get by this argument that
\begin{equation}\label{eq: telescope estimate}
\| x_m - x_n \| \le \sum_{k=n}^m \| x_{k+1} - x_k \| \le \sum_{k = n}^m \frac{\Lambda }{\sigma \theta } (   f_k^{\theta}  - f_{k+1}^{\theta }  ) = \frac{\Lambda }{\sigma \theta } (   f_n^{\theta}  - f_m^{\theta }  ) \le \frac{\Lambda }{\sigma \theta } f_n^{\theta}.
\end{equation}
Since $x^*$ is an accumulation point, we can pick $n$ so large that (recall that $f$ is continuous and $f(x^*) = 0$)
\[
 \|x_n - x^*\| < \frac{\epsilon}{2} \quad \text{and} \quad \frac{\Lambda}{\sigma \theta}  f_n^{\theta}  < \frac{\epsilon}{2}.
\]
Then~\eqref{eq: telescope estimate} inductively implies $\| x_m - x^*\| < \epsilon$ for all $m \ge n$. This proves that $x^*$ is the limit point of the sequence, and, by~\eqref{A3}, $g_n^- \to 0$. 

To estimate the convergence rate, let $r_n = \sum_{k = n}^{\infty} \| x_{k+1} - x_k\|$. Then $\|x_n - x^*\| \le r_n$, so it suffices to estimate the latter. By~\eqref{eq: telescope estimate},~\eqref{eq: Loj ineq in proof}, and~\eqref{A3}, there exists $n_0 \ge 1$ such that for $n \ge n_0$ it holds that
\[
r_n^{\frac{1 - \theta}{\theta}}
\le \left( \frac{\Lambda}{\sigma \theta}\right)^{\frac{1 - \theta}{\theta}}  f_n^{1 - \theta}
\le \left( \frac{\Lambda}{\sigma \theta}\right)^{\frac{1 - \theta}{\theta}}  \frac{\Lambda}{\kappa} \|x_{n+1} - x_n\| 
= \left( \frac{\Lambda}{\sigma \theta}\right)^{\frac{1 - \theta}{\theta}}  \frac{\Lambda}{\kappa} (r_n - r_{n+1}),
\]
that is,
\begin{equation}\label{eq: recursive estimate}
r_{n+1} \le r_n - \nu r_n^{\frac{1-\theta}{\theta}}
\end{equation}
with $\nu = ( \frac{\Lambda}{\sigma \theta})^{\frac{\theta-1}{\theta}}  \frac{\kappa}{\Lambda}$. Now, if $\theta = 1/2$, we get from~\eqref{eq: recursive estimate} that $\nu \in (0,1)$, and
\[
r_n \le r_{n_0} (1 -\nu)^{n-n_0}  \left(e^{\ln (1 - \nu)}\right)^n
\]
for $n \ge n_0$. The case $0 < \theta < 1/2$ is more delicate. We follow Levitt~\cite{Levitt2012}: put $p = \frac{\theta}{1 - 2\theta}$, $C \ge \max((\frac{\nu}{p})^{-p}, r_{n_0} n_0^{-p})$, and $s_n = Cn^{-p}$; then $s_{n_0} \ge r_{n_0}$, and
\[
s_{n+1} = s_n(1 + n^{-1})^{-p} \ge s_n(1 - pn^{-1}) = s_n - \frac{p}{C^{1/p}}s_n^{\frac{p+1}{p}} \ge s_n - \nu s_n^{\frac{p+1}{p}} = s_n - \nu s_n^{\frac{1-\theta}{\theta}}
\]
(the first inequality holding by convexity of $x^{-p}$). Using induction, it now follows from~\eqref{eq: recursive estimate} that $r_n \le s_n$ for all $n \ge n_0$, which finishes the proof.\qquad\endproof

\bibliographystyle{siam}
\bibliography{lojasiewicz}

\providecommand{\noopsort}[1]{}
\begin{thebibliography}{10}

\bibitem{AbsilMahonyAndrews2005}
{\sc P.-A. Absil, R.~Mahony, and B.~Andrews}, {\em Convergence of the iterates
  of descent methods for analytic cost functions}, SIAM J. Optim., 16 (2005),
  pp.~531--547.

\bibitem{AbsilMahonySepulchre2008}
{\sc P.-A. Absil, R.~Mahony, and R.~Sepulchre}, {\em Optimization {A}lgorithms
  on {M}atrix {M}anifolds}, Princeton University Press, Princeton, NJ, 2008.

\bibitem{AbsilMalick2012}
{\sc P.-A. Absil and J.~Malick}, {\em Projection-like retractions on matrix
  manifolds}, SIAM J. Optim., 22 (2012), pp.~135--158.

\bibitem{AttouchBolte2009}
{\sc H.~Attouch and J.~Bolte}, {\em On the convergence of the proximal
  algorithm for nonsmooth functions involving analytic features}, Math.
  Program., 116 (2009), pp.~5--16.

\bibitem{Attouchetal2010}
{\sc H.~Attouch, J.~Bolte, P.~Redont, and A.~Soubeyran}, {\em Proximal
  alternating minimization and projection methods for nonconvex problems: {A}n
  approach based on the {Kurdyka-{\L{}}ojasiewicz} inequality}, Math. Oper.
  Res., 35 (2010), pp.~438--457.

\bibitem{AttouchBolteSvaiter2013}
{\sc H.~Attouch, J.~Bolte, and B.F. Svaiter}, {\em Convergence of descent
  methods for semi-algebraic and tame problems: {P}roximal algorithms,
  forward-backward splitting, and regularized {G}auss-{S}eidel methods}, Math.
  Program., 137 (2013), pp.~91--129.

\bibitem{BolteDaniilidisLewis2007}
{\sc J.~Bolte, A.~Daniilidis, and A.~Lewis}, {\em The {{\L{}}ojasiewicz}
  inequality for nonsmooth subanalytic functions with applications to
  subgradient dynamical systems}, SIAM J. Optim., 17 (2007), pp.~1205--1223.

\bibitem{Bolteetal2007}
{\sc J.~Bolte, A.~Daniilidis, A.~Lewis, and M.~Shiota}, {\em Clarke
  subgradients of stratifiable functions}, SIAM J. Optim., 18 (2007),
  pp.~556--572.

\bibitem{Bolteetal2010}
{\sc J.~Bolte, A.~Daniilidis, O.~Ley, and L.~Mazet}, {\em Characterizations of
  {{\L{}}}ojasiewicz inequalities: {S}ubgradient flows, talweg, convexity},
  Trans. Amer. Math. Soc., 362 (2010), pp.~3319--3363.

\bibitem{BunseGerstneretal1991}
{\sc A.~Bunse-Gerstner, R.~Byers, V.~Mehrmann, and N.K. Nichols}, {\em
  Numerical computation of an analytic singular value decomposition of a matrix
  valued function}, Numer. Math., 60 (1991), pp.~1--39.

\bibitem{BurerMonteiro2003}
{\sc S.~Burer and R.D.C. Monteiro}, {\em A nonlinear programming algorithm for
  solving semidefinite programs via low-rank factorization}, Math. Program., 95
  (2003), pp.~329--357.

\bibitem{CancesEhrlacherLelievre2014}
{\sc E.~Canc{\`e}s, V.~Ehrlacher, and T.~Leli{\`e}vre}, {\em Greedy algorithms
  for high-dimensional eigenvalue problems}, Constr. Approx., 40 (2014),
  pp.~387--423.

\bibitem{CasonAbsilVanDooren2013}
{\sc T.P. Cason, P.-A. Absil, and P.~Van~Dooren}, {\em Iterative methods for
  low rank approximation of graph similarity matrices}, Linear Algebra Appl.,
  438 (2013), pp.~1863--1882.

\bibitem{DieudonneIII}
{\sc J.~Dieudonn{\'e}}, {\em Treatise on {A}nalysis. {V}ol. {III}}, Academic
  Press, New York, London, 1972.

\bibitem{Guignard1969}
{\sc M.~Guignard}, {\em Generalized {K}uhn--{T}ucker conditions for
  mathematical programming problems in a {B}anach space}, SIAM J. Control, 7
  (1969), pp.~232--241.

\bibitem{HackbuschKuehn2009}
{\sc W.~Hackbusch and S.~K{\"u}hn}, {\em A new scheme for the tensor
  representation}, J. Fourier Anal. Appl., 15 (2009), pp.~706--722.

\bibitem{HarauxJendoubi2011}
{\sc A.~Haraux and M.A. Jendoubi}, {\em The {{\L{}}}ojasiewicz gradient
  inequality in the infinite-dimensional {H}ilbert space framework}, J. Funct.
  Anal., 260 (2011), pp.~2826--2842.

\bibitem{HelmkeShayman1995}
{\sc U.~Helmke and M.A. Shayman}, {\em Critical points of matrix least squares
  distance functions}, Linear Algebra Appl., 215 (1995), pp.~1--19.

\bibitem{Huang2006}
{\sc S.-Z. Huang}, {\em Gradient {I}nequalities. {W}ith {A}pplications to
  {A}symptotic {B}ehavior and {S}tability of {G}radient-like {S}ystems},
  American Mathematical Society, Providence, RI, 2006.

\bibitem{Journeeetal2010}
{\sc M.~Journ{\'e}e, F.~Bach, P.-A. Absil, and R.~Sepulchre}, {\em Low-rank
  optimization on the cone of positive semidefinite matrices}, SIAM J. Optim.,
  20 (2010), pp.~2327--2351.

\bibitem{KochLubich2007}
{\sc O.~Koch and C.~Lubich}, {\em Dynamical low-rank approximation}, SIAM J.
  Matrix Anal. Appl., 29 (2007), pp.~434--454.

\bibitem{KochLubich2010}
\leavevmode\vrule height 2pt depth -1.6pt width 23pt, {\em Dynamical tensor
  approximation}, SIAM J. Matrix Anal. Appl., 31 (2010), pp.~2360--2375.

\bibitem{KoldaBader2009}
{\sc T.G. Kolda and B.W. Bader}, {\em Tensor decompositions and applications},
  SIAM Rev., 51 (2009), pp.~455--500.

\bibitem{KrantzParks2002}
{\sc S.G. Krantz and H.R. Parks}, {\em A {P}rimer of {R}eal {A}nalytic
  {F}unctions}, Birkh\"auser Boston Boston, 2nd~ed., 2002.

\bibitem{KressnerSteinlechnerVandereycken2013}
{\sc D.~Kressner, M.~Steinlechner, and B.~Vandereycken}, {\em Low-rank tensor
  completion by {R}iemannian optimization}, BIT, 54 (2014), pp.~447--468.

\bibitem{Kurdyka1998}
{\sc K.~Kurdyka}, {\em On gradients of functions definable in o-minimal
  structures}, Ann. Inst. Fourier (Grenoble), 48 (1998), pp.~769--783.

\bibitem{Lageman2002}
{\sc C.~Lageman}, {\em {K}onvergenz reell-analytischer gradienten\"ahnlicher
  {S}ysteme}, diploma thesis, Universit{\"a}t W{\"u}rzburg, W{\"u}rzburg,
  Germany, 2002.
\newblock In German.

\bibitem{Lageman2007a}
\leavevmode\vrule height 2pt depth -1.6pt width 23pt, {\em Convergence of
  {G}radient-like {D}ynamical {S}ystems and {O}ptimization {A}lgorithms}, PhD
  thesis, Universit{\"a}t W{\"u}rzburg, W{\"u}rzburg, Germany, 2007.

\bibitem{Lageman2007}
\leavevmode\vrule height 2pt depth -1.6pt width 23pt, {\em Pointwise
  convergence of gradient-like systems}, Math. Nachr., 280 (2007),
  pp.~1543--1558.

\bibitem{Levitt2012}
{\sc A.~Levitt}, {\em Convergence of gradient-based algorithms for the
  {H}artree-{F}ock equations}, ESAIM Math. Model. Numer. Anal., 46 (2012),
  pp.~1321--1336.

\bibitem{LiUschmajewZhang2015}
{\sc Z.~Li, A.~Uschmajew, and S.~Zhang}, {\em On convergence of the maximum
  block improvement method}, SIAM J. Optim., 25 (2015), pp.~210--233.

\bibitem{Lojasiewicz1965}
{\sc S.~\L{}ojasiewicz}, {\em Ensemble semi-analytique}.
\newblock Note des cours, Institut des Hautes Etudes Scientifique, 1965.

\bibitem{LubichOseledets2013}
{\sc C.~Lubich and I.V. Oseledets}, {\em A projector-splitting integrator for
  dynamical low-rank approximation}, BIT, 54 (2014), pp.~171--188.

\bibitem{Lubichetal2013}
{\sc C.~Lubich, T.~Rohwedder, R.~Schneider, and B.~Vandereycken}, {\em
  Dynamical approximation by hierarchical {T}ucker and tensor-train tensors},
  SIAM J. Matrix Anal. Appl., 34 (2013), pp.~470--494.

\bibitem{Luke2013}
{\sc D.R. Luke}, {\em Prox-regularity of rank constraints sets and implications
  for algorithms}, J. Math. Imaging Vision, 47 (2013), pp.~231--238.

\bibitem{MerletNguyen2013}
{\sc B.~Merlet and T.N. Nguyen}, {\em Convergence to equilibrium for
  discretizations of gradient-like flows on {R}iemannian manifolds},
  Differential Integral Equations, 26 (2013), pp.~571--602.

\bibitem{Mishraetal2013}
{\sc B.~Mishra, G.~Meyer, F.~Bach, and R.~Sepulchre}, {\em Low-rank
  optimization with trace norm penalty}, SIAM J. Optim., 23 (2013),
  pp.~2124--2149.

\bibitem{Mishraetal2014}
{\sc B.~Mishra, G.~Meyer, S.~Bonnabel, and R.~Sepulchre}, {\em Fixed-rank
  matrix factorizations and {R}iemannian low-rank optimization}, Comput.
  Statist., 29 (2014), pp.~591--621.

\bibitem{NocedalWright2006}
{\sc J.~Nocedal and S.J. Wright}, {\em Numerical {O}ptimization}, Springer, New
  York, 2006.

\bibitem{Noll2014}
{\sc D.~Noll}, {\em Convergence of non-smooth descent methods using the
  {K}urdyka-{{\L{}}}ojasiewicz inequality}, J. Optim. Theory Appl., 160 (2014),
  pp.~553--572.

\bibitem{NonnenmacherLubich2008}
{\sc A.~Nonnenmacher and C.~Lubich}, {\em Dynamical low-rank approximation:
  {A}pplications and numerical experiments}, Math. Comput. Simulation, 79
  (2008), pp.~1346--1357.

\bibitem{Oseledets2011}
{\sc I.V. Oseledets}, {\em Tensor-train decomposition}, SIAM J. Sci. Comput.,
  33 (2011), pp.~2295--2317.

\bibitem{OSheaWilson2004}
{\sc D.B. O'Shea and L.C. Wilson}, {\em Limits of tangent spaces to real
  surfaces}, Amer. J. Math., 126 (2004), pp.~951--980.

\bibitem{RockafellarWets1998}
{\sc R.T. Rockafellar and R.J.-B. Wets}, {\em Variational {A}nalysis},
  Springer-Verlag, Berlin, 1998.

\bibitem{ShalitWeinshallChechik2012}
{\sc U.~Shalit, D.~Weinshall, and G.~Chechik}, {\em Online learning in the
  embedded manifold of low-rank matrices}, J. Mach. Learn. Res., 13 (2012),
  pp.~429--458.

\bibitem{Shub1986}
{\sc M.~Shub}, {\em Some remarks on dynamical systems and numerical analysis},
  in Dynamical {S}ystems and {P}artial {D}ifferential {E}quations ({C}aracas,
  1984), Univ. Simon Bolivar, Caracas, 1986, pp.~69--91.

\bibitem{Tanetal2014}
{\sc M.~Tan, I.~W. Tsang, L.~Wang, B.~Vandereycken, and S.~J. Pan}, {\em
  Riemannian pursuit for big matrix recovery}, in Proceedings of the 31st
  International Conference on Machine Learning (ICML), vol.~32 of JMLR Workshop
  and Conference Proceedings, 2014, pp.~1539--1547.

\bibitem{UschmajewVandereycken2013}
{\sc A.~Uschmajew and B.~Vandereycken}, {\em The geometry of algorithms using
  hierarchical tensors}, Linear Algebra Appl., 439 (2013), pp.~133--166.

\bibitem{UschmajewVandereycken2014}
\leavevmode\vrule height 2pt depth -1.6pt width 23pt, {\em Line-search methods
  and rank increase on low-rank matrix varieties}, in Proceedings of the 2014
  International Symposium on Nonlinear Theory and its Applications (NOLTA2014),
  2014, pp.~52--55.

\bibitem{Vandereycken2013}
{\sc B.~Vandereycken}, {\em Low-rank matrix completion by {R}iemannian
  optimization}, SIAM J. Optim., 23 (2013), pp.~1214--1236.

\bibitem{VandereyckenVandewalle2010}
{\sc B.~Vandereycken and S.~Vandewalle}, {\em A {R}iemannian optimization
  approach for computing low-rank solutions of {L}yapunov equations}, SIAM J.
  Matrix Anal. Appl., 31 (2010), pp.~2553--2579.

\bibitem{XuYin2013}
{\sc Y.~Xu and W.~Yin}, {\em A block coordinate descent method for regularized
  multiconvex optimization with applications to nonnegative tensor
  factorization and completion}, SIAM J. Imaging Sci., 6 (2013),
  pp.~1758--1789.

\end{thebibliography}

\end{document}